\documentclass[11pt,a4paper,fleqn]{article}
\input{packages}
\newcommand{\IV}{\mathrm{var}}

\definecolor{dgreen}{rgb}{0,0.5,0}
\definecolor{dblue}{rgb}{0,0,0.5}
\definecolor{dred}{rgb}{0.6,0.0,0.1}
\definecolor{dgold}{rgb}{0.5,0.3,0.0}
\definecolor{dvio}{rgb}{0.6,0.3,0.5}
\definecolor{gray}{rgb}{0.5,0.5,0.5}
\definecolor{dbraun}{rgb}{.5,0.2,0}

\newcommand{\colre}{dred}

\newcommand{\colrem}{dgold}
\newcommand{\colil}{dgreen}
\newcommand{\Vnorm}[2][]{\lVert#2\rVert_{#1}}   

\newcommand{\VnormInf}[2][\infty]{\Vnorm[{#1}]{#2}} 
\def\var{\mathop{\mathrm{var}}\nolimits}%
\def\argmin{\mathop{\mathrm{arg \; min}}\limits}%
\def\liminf{\mathop{\mathrm{lim \, inf}}\limits}%
\newcommand{\lra}{\longrightarrow} 

\newcommand{\eps}{\varepsilon}
\renewcommand{\subset}{\subseteq}
\newcommand{\IN}{\mathbb{N}}

\newcommand{\IR}{\mathbb{R}}

\newcommand{\IP}{\mathbb{P}}
\newcommand{\IE}{\mathbb{E}}
\newcommand{\iid}{\overset{\text{iid}}{\sim}}

\newcommand{\lcb}{\left\lbrace} 
\newcommand{\rcb}{\right\rbrace} 
\newcommand{\la}{\langle} 
\newcommand{\ra}{\rangle} 
\newcommand{\lv}{\left\vert} 
\newcommand{\rv}{\right\vert} 
\newcommand{\lV}{\left\Vert} 
\newcommand{\rV}{\right\Vert} 
\newcommand{\lb}{\left(} 
\newcommand{\rb}{\right)} 
\newcommand*{\mc}[1]{\mathcal{#1}}

\newcommand{\dif}{\text{d}}

\declaretheoremstyle[
    spaceabove=10pt, 
    spacebelow=6pt, 
    headfont=\color{\colre}\normalfont\bfseries,
    notefont=\mdseries\bfseries, 
    notebraces={(}{)}, 
    bodyfont=\normalfont,
    postheadspace=.3em,
    headpunct={.},
     mdframed={
    	backgroundcolor=gray!10, 
    	linecolor=gray!07, 
    	innertopmargin=6pt,
    	roundcorner=5pt, 
    	innerbottommargin=10pt, 
    	skipabove=\parsep, 
    	skipbelow=\parsep}]{restyle}

\declaretheoremstyle[
    spaceabove=8pt, 
    spacebelow=8pt, 
    headfont=\color{\colrem}\normalfont\bfseries,
    notefont=\mdseries\bfseries, 
    notebraces={(}{)}, 
    bodyfont=\normalfont\itshape,
    postheadspace=.3em,
    qed=\smaller$\color{\colrem}\square$, 
    headpunct={.}]{remstyle}

\declaretheoremstyle[
    spaceabove=8pt, 
    spacebelow=8pt, 
    headfont=\color{\colil}\normalfont\bfseries,
    notefont=\mdseries\bfseries, 
    notebraces={(}{)}, 
    bodyfont=\normalfont,
    postheadspace=.3em,
    qed=\smaller$\color{\colil}\square$, 
    headpunct={.}]{ilstyle}

 \declaretheoremstyle[
 spaceabove=8pt, 
 spacebelow=8pt, 
 headfont=\color{\colil}\normalfont\bfseries,
 notefont=\mdseries\bfseries, 
 notebraces={(}{)}, 
 bodyfont=\normalfont,
 postheadspace=.3em,
 headpunct={.}, 
 mdframed={
 	backgroundcolor=gray!13, 
 	linecolor=gray!10, 
 	innertopmargin=6pt,
 	roundcorner=5pt, 
 	innerbottommargin=10pt, 
 	skipabove=\parsep, 
 	skipbelow=\parsep } ]{exstyle}
  
\declaretheorem[name=Theorem, style=restyle, numberwithin=section]{theorem}
\declaretheorem[name=Theorem, style=restyle, numbered=no]{theorem*}
\declaretheorem[name=Example, style=restyle, numberlike=theorem]{example}
\declaretheorem[name=Example, style=exstyle, numbered=no]{example*}
\declaretheorem[name=Definition, style=restyle, numberlike = theorem]{definition}
\declaretheorem[name=Definition, style=restyle, numbered=no]{definition*}
\declaretheorem[name=Reminder, style=restyle, numberlike = theorem]{reminder}
\declaretheorem[name=Lemma, style=restyle, numberlike=theorem]{lemma}
\declaretheorem[name=Proposition, style=restyle, numberlike=theorem]{proposition}
\declaretheorem[name=Proposition, style=restyle, numbered=no]{proposition*}
\declaretheorem[name=Corollary, style=restyle, numberlike=theorem]{corollary}
\declaretheorem[name=Remark, style=remstyle, numberlike=theorem]{remark}
\declaretheorem[name=Assumption, style=restyle, numberlike=theorem]{assumption}

\newcommand{\Nz}{{\mathbb N}}

\newcommand{\Rz}{{\mathbb R}}

\renewcommand{\asymp}{\sim}

\newcommand{\xden}[1][]{f_{#1}}

\newcommand{\sera}{\rho}

\newcommand{\msera}[1][]{\sera_\star}

\newcommand{\bias}{\mathrm{bias}}

\newcommand{\sPara}{\beta}

\newcommand{\qden}{\mathbbm{q}}

\newcommand{\privmech}{\mathbbm{Q}}
\newcommand{\Laplace}{\text{Laplace}}

\makeatletter
\def\namedlabel#1#2{\begingroup
	#2%
	\def\@currentlabel{#2}%
	\phantomsection\label{#1}\endgroup
}

\begin{document}
\author{{\sc Sandra Schluttenhofer}\,\thanks{Aarhus Universitet, Institut for Matematik, Ny Munkegade 118, DK-8000 Aarhus C, Denmark, e-mail:
		\url{schluttenhofer@math.au.dk}} \and {\sc Jan Johannes}\,\thanks{Ruprecht-Karls-Universit\"at Heidelberg, Institut f\"ur Angewandte
		Mathematik, M$\Lambda$THEM$\Lambda$TIKON, Im Neuenheimer Feld 205,
		D-69120 Heidelberg, Germany, e-mail:
		\url{johannes@math.uni-heidelberg.de}}}
\title{Adaptive pointwise density estimation under local differential privacy} 
\maketitle
\begin{abstract}
  We consider the estimation of a density at a fixed point under a
  local differential privacy constraint, where the observations are
  anonymised before being available for statistical inference. We
  propose both a privatised version of a projection density estimator
  as well as a kernel density estimator and derive their minimax rates
  under a privacy constraint. There is a twofold deterioration of the
  minimax rates due to the anonymisation, which we show to be
  unavoidable by providing lower bounds. In both estimation procedures
  a tuning parameter has to be chosen. We suggest a variant of the
  classical Goldenshluger-Lepski method for choosing the bandwidth and
  the cut-off dimension, respectively, and analyse its performance. It
  provides adaptive minimax-optimal (up to $\log$-factors)
  estimators. We discuss in detail how the lower and upper bound
  depend on the privacy constraints, which in turn is reflected by a
  modification of the adaptive method.
\end{abstract}

\section{Introduction}
In this paper we are interested in estimating the value $f(t)$ of an unknown
density $f$ at a fixed point $t$ under a local
differential privacy constraint, where
the raw sample  of independent and identically distributed (iid) real random variables
\begin{align}
	\label{priv:raw:sample}
	X_i \iid \xden, \qquad i \in \lcb 1 , \dotsc, n \rcb,
\end{align}
is anonymised before being available for statistical inference.  The
paper is organised as follows: in \cref{sec:review} we review
pointwise density estimation and local differential privacy. We focus
on a privatised kernel and projection density estimator in
\cref{sec:KDE} and \cref{sec:PDE}, respectively. In both cases the
non-private estimator based on the raw sample equals a mean
$\tfrac{1}{n}\sum_{i=1}^ng(X_i)$ for a function $g$ depending on the
evaluation point $t$ and a tuning parameter. We study a privatised
version of these estimators, where the $i$-th data holder releases a
sanitised version of $g(X_i)$ rather than $X_i$. The anonymisation is
obtained by a general Laplace perturbation which guarantees local
differential privacy in both
cases.  We derive an upper bound for their
mean squared error and show a matching lower bound for the maximal
mean squared error over Hölder classes when taking the infimum over
all possible estimators and all local differential privatisation
methods.  The estimators attain the lower bound (up to a constant)
only if the tuning parameter is chosen optimally. Since the necessary
information for their optimal choice is typically inaccessible in
practice, we propose a fully data-driven choice inspired by the work
of \cite{GoldenshlugerLepski2011}. We establish an oracle inequality
for the completely data-driven privatised kernel and projection
density estimator. Comparing the upper bounds of their mean squared errors
with the minimax rate of convergence they feature an additional
logarithmic factor, which possibly results in a deterioration of the
rate. The appearance of the logarithmic factor within the rate is a
known phenomenon in the context of pointwise estimation and it is widely
considered as an acceptable price for adaptation
(cf. \cite{BrownLow1996} in the context of non-parametric Gaussian
regression or \cite{LaurentLudenaPrieur2008} given direct Gaussian
observations). In the appendix we
gather elementary bounds for sums of iid real random variables.

\section{Review on point-wise density estimation and  local differential privacy}\label{sec:review}

\paragraph*{Pointwise density estimation.}
In a classical density estimation problem we have $n$ real-valued
observations $(X_i)_{i \in \lcb 1, \dots, n \rcb}$, whose underlying
distribution has a Lebesgue density $f$. For a fixed point $t \in\IR$ one aims to estimate $f(t) \in \IR$. In this setting there exists a vast amount of estimators in the literature. We focus on two classical approaches, first we introduce kernel density estimators and secondly, projection density estimators. Let $K$ be a kernel function, as usual we assume $K$ to be square-integrable and bounded and to integrate to $1$. The kernel density estimator for $f(t)$ is given by
\begin{align}
	\label{KDE}
	\tilde{f}_h(t) := \frac{1}{n} \sum_{i=1}^n K_h(X_i -t),
\end{align}
where we use the notation $K_h(x) := \tfrac{1}{h} K\lb \tfrac{x}{h}\rb, x \in \IR$, and $h > 0$ is a bandwidth. If $f$ is a square-integrable density supported on $[0,1]$ an alternative methods uses the projection density estimator. More precisely, for $d \in \IN$ let $(\varphi_j)_{j \in \lcb 1, \dots, d \rcb}$ be an orthonormal system in $\mc L^2([0,1])$, the space of all square-integrable real-valued functions defined on $[0,1]$. The projection density estimator for $f(t)$ is given by 
\begin{align}
	\label{PDE}
	\tilde{f}_d(t) = \sum_{j=1}^d \tilde{a}_j \varphi_j(t), \qquad \tilde{a}_j = \frac{1}{n} \sum_{i=1}^n \varphi_j(X_i).
\end{align} 
 Both estimation strategies depend on a tuning parameter, the bandwidth $h > 0$ respectively the truncation parameter $d \in \IN$. 
 Typically, the performance of an estimator $T_n$ of $f(t)$ is assessed by considering its estimation risk or mean squared error given by $\IE \lb \lv T_n - f(t) \rv^2 \rb$. Upper and lower bounds for the maximal estimation risk over different classes of densities have been studied extensively in the literature. It is well-known that the accuracy of both kernel and projection density estimators crucially depends on the choice of the tuning parameters. The data-driven choice of these parameters is subject of considerable literature. A well-studied data-driven approach for choosing the tuning parameters is the Goldenshluger-Lepski method (introduced in \cite{GoldenshlugerLepski2011}), which is based on finding an estimate of the mean squared error for each $h$ resp. $D$ in a collection and minimizing the estimate with respect to the smoothing parameters.  We refer to the textbooks \cite{Tsybakov2009} and \cite{Comte2017} and the references therein.

\paragraph{Local differential privacy.}  In cases where the
observations contain sensitive private information, the \textit{raw}
sample $(X_i)_{i \in \lcb 1, \dots, n \rcb}$ is unavailable to the
statistician. Instead, the data holder releases a \textit{privatized} or \textit{sanitized} sample
	$(Z_i)_{i \in \lcb 1, \dots, n \rcb}$ 
that is obtained from $(X_i)_{i \in \lcb1, \dots, n \rcb}$ by a
stochastic transformation $\privmech$, called \textit{privacy
  mechanism}, \textit{stochastic channel} or \textit{data release mechanism}. In the computer science literature, the samples $(X_i)_{i \in \lcb 1, \dots, n \rcb}$ and $(Z_i)_{i \in \lcb 1, \dots, n \rcb}$ are often called \textit{databases}. 
Formally, let  $X$ and $Z$ be defined on a common probability space with values in measurable spaces $(\mc X, \mathscr X)$ and $(\mc Z, \mathscr Z)$, respectively. A local privacy mechanism $\privmech$ is associated with a Markov kernel $\kappa_\privmech: (\mc X, \mathscr X) \lra [0,1]$ with $\kappa_\privmech(x, B) = \mathbb{P}(Z \in B \mid X = x) = \privmech(B \mid x)$ for all $x \in \mc X$ and $B \in \mathscr{Z}$. In other words, the privacy mechanism $\privmech$ is the regular conditional distribution of $Z$ given $X$. We assume that the stochastic channel satisfies a privacy constraint, which we formalize next. 

\begin{definition}[$\alpha$-differential privacy]
We call $Z$ a \textbf{$\alpha$-differentially private view }of $X$ with privacy parameter $\alpha \geq 0$ if the conditional distribution $\privmech$ satisfies
	\begin{align}
		\label{priv:condition}
		{\privmech( B \mid x)} \leq \exp(\alpha) \cdot \privmech(B \mid x') \qquad \qquad \text{for all } B \in \mathscr Z \text{ and } x, x' \in \mc X.
	\end{align}
The privacy mechanism $\privmech$ is then called \textbf{$\alpha$-differentially locally private}. We denote the set of all $\alpha$-differentially locally private mechanism by $\mc Q_\alpha$.  
\end{definition} 

The sample $(Z_i)_{i \in \lcb 1, \dots, n \rcb}$ generated with
$\privmech$ satisfying \eqref{priv:condition} is called a
\textbf{$\alpha$-differentially locally private (non-interactive)}
view of the raw sample $(X_i)_{i \in \lcb 1, \dots, n \rcb}$ in
\eqref{priv:raw:sample}. The term \textbf{locally} refers to the fact
that for the generation of the $i$th sanitized observation $Z_i$ the
data holder only requires the $i$th raw observation $X_i$, thus, the raw data can be stored locally. In contrast to this, there also exists the concept of \textbf{global} differential privacy, where a data collector is entrusted with the data and generates a privatized database $(Z_i)_{i \in \lcb 1, \dots, n \rcb}$ based on the entire raw data set $(X_i)_{i \in \lcb 1, \dots, n \rcb}$. The privatization is called \textbf{interactive} if the $i$th data holder also has access to (already generated) sanitized observations $(Z_k)_{k \in \lcb 1, \dots, i -1 \rcb}$. In the non-interactive case the data holders do not need to interact with each other in order to generate the private views. 
\paragraph{Related literature.} The concept of differential privacy was essentially introduced in the series of papers \cite{DinurNissim2003}, \cite{DworkNissim2004} and \cite{Dwork2006}. \cite{Dwork2008} gives an overview of the early results in the field. First statistical results are derived in \cite{WassermanZhou2010} and \cite{HallRinaldoWasserman2013}, where both papers work under global privacy constraints. \cite{DuchiJordanWainwright2018} provide a toolbox of methods for deriving minimax rates of estimation under a local privacy constraint. Naturally, there exist many more concepts of privacy (smooth privacy, divergence-based privacy, approximate privacy to mention but a few), for a broad overview we refer the reader to \cite{BarberDuchi2014}.  \\

Let us now first heuristically explain the implications of the condition \eqref{priv:condition}. A small value of $\alpha$ (close to $0$) corresponds to a high privacy guarantee. In the extreme case $\alpha=0$ the privacy mechanism $\privmech$ does not depend on the value of the input data $X$, in other words $Z$ and $X$ are independent. Hence, we achieve total privacy. Naturally, the privatized sample is useless for making inference on the distribution of $X$. Large values of $\alpha$ allow for low privacy, since a change in the original observation can then yield a completely different distribution for the output random variable and it is thus easier to draw conclusions about the raw data. Let us now formalize the effect \eqref{priv:condition} has on the information about concrete input data points.
Assume we want to find out whether the original (raw) data comes from
data holder 1 (with value $x$ with associated probability $\IP_0 =
\privmech(\cdot \mid x)$) or from data holder 2 (with value $x' \neq x$ with associated probability $\IP_1 =  \privmech(\cdot \mid x')$). This simple two-point testing task is solved using the Neyman-Pearson-Lemma. The privacy constraint gives a bound for the maximal power a test can achieve. The following proposition is a reformulation of Theorem 2.4. in \cite{WassermanZhou2010} and we state its proof in our setting for completeness. 
\begin{proposition}[Plausible deniability]
	\label{priv:discovery}
	Let $Z$ be an $\alpha$-differentially private view of $X$ obtained through the channel $\privmech$. Let $x \neq x'$. Any level-$\gamma$-test based on the observation $Z$ and the channel $\privmech$ for the task
 \begin{align*}
 	H_0: \lcb \IP_0  = \privmech(\cdot \mid x) \rcb \qquad \text{ against } \qquad H_1: \lcb \IP_1 = \privmech(\cdot \mid x') \rcb  
 \end{align*}
has power bounded by $\gamma \exp(\alpha)$. 
\end{proposition}
\begin{proof}[Proof of \cref{priv:discovery}]
	The Neyman-Pearson Lemma states that the highest possible power (i.e.\ minimal type II error probability) is obtained by a test of the form
	\begin{align*}
		\Delta := \mathds{1}_{\lcb \dif \IP_1 \geq \tau \, \dif \IP_0 \rcb},
	\end{align*}
	where $\dif \IP_1$ and $\dif \IP_0$ are densities with respect to an arbitrary dominating measure and the threshold $\tau$ is such that
	\begin{align*}
		\IP_0(\Delta = 1) = \IP_0\lb \frac{\dif \IP_1}{\dif \IP_0} \geq \tau \rb \leq \gamma.
	\end{align*}
	Note that the distributions $\IP_0 = \privmech(\cdot \mid x)$ and $\IP_1 = \privmech(\cdot \mid x')$ satisfy $\IP_1(A) \leq \exp(\alpha) \IP_0(A)$	for any measurable set $A$. Hence, the power of the test is bounded by
	\begin{align*}
		1 - \IP_1(\Delta = 0)  = \IP_1(\Delta = 1) \leq \exp(\alpha) \IP_0(\Delta = 1) \leq \exp(\alpha) \gamma,
	\end{align*}
which proves the result.
\end{proof}
We give two popular examples of privacy mechanisms that satisfy the privacy constraint \eqref{priv:condition}.%
\begin{example}[Perturbation approach, "Adding noise"]
	\label{ex:adding:noise}
 The perturbation approach consists of adding centred noise $\xi$ with Lebesgue density $h$ to the observations, i.e.\
	\begin{align*}
		Z := X + \xi \qquad \text{ with } \xi \sim h, \quad \IE \xi = 0.
	\end{align*}
Then the stochastic channel $\privmech$ has the density $	\qden(z \mid x) = h(z-x)$ 
with respect to the Lebesgue measure. 
The most popular noise density is the Laplace density with appropriately chosen variance. 
\end{example}

\begin{remark}[Naive privatization methods]
	\label{naive}
	As described in \cref{ex:adding:noise} a standard technique for privatizing data is to add Laplace noise directly to the observations. Inference on the density $f$ of $X$ based on observations of $X+\xi$ with privatization noise $\xi \sim h$ is essentially a deconvolution problem, since the density of $X + \xi$ equals the convolution of $f$ and $h$.  Minimax rates for the pointwise density estimation in this setting are well-known and depend on the regularity of both the density of interest $f$ and the added noise $\xi$.  By choosing a slightly more elaborate privacy mechanism we are able to significantly improve the private minimax rate.  
\end{remark}
Let us now consider a more general form of a privacy mechanism.
\begin{example}[Exponential mechanism]
	\label{priv:ex:exponential}
	Let $\zeta: \mc X \times \mc Z \longrightarrow [0,\infty)$ be a function and define the sensitivity of $\zeta$ by
	\begin{align*}
		\delta := \sup_{x,y \in \mc X} \sup_{z \in Z} \lv \zeta(x,z) - \zeta(y,z) \rv
	\end{align*}
as the maximal change of  $\zeta$ that can occur due to altering the input data. Define the density
\begin{align*}
	h(z \mid x) = \frac{\exp(-\alpha \frac{\zeta(x,z)}{2 \delta})}{\int \exp(-\alpha \frac{\zeta(x,t)}{2 \delta}) \dif t }
\end{align*}
and sample $Z \sim h(\cdot \mid x)$. \cite{McSherryTalwar2007} show that the exponential mechanism yields a $\alpha$-differentially private channel. 
\end{example}

Consider $\mc X = [0,1]$. Let us point out that for $\zeta(x,z) = \lv x - z \rv$ (and $\delta = 2$) the exponential mechanism in \cref{priv:ex:exponential} corresponds to the perturbation approach in \cref{ex:adding:noise} with $\xi \sim \text{Laplace}(0, \tfrac{4}{\alpha})$. We briefly recall the Laplace distribution.
\begin{reminder}[Laplace distribution]
	\label{priv:lap:reminder}
	With $\xi \sim \Laplace(0,1)$ we denote the distribution with probability density 
	$	 	x \mapsto \tfrac{1}{2} \exp\lb-\lv x\rv \rb$.
	We have $	\IE \lv \xi \rv^m = m!$ and    $b \xi + \mu \sim \Laplace(\mu, b)$ for $\mu \in \IR$, $b  > 0$. 
\end{reminder} 

Thus, we have seen that perturbation with a Laplace distribution yields $\alpha$-differential privacy.  By more generally choosing $\zeta(x,z) = \lv g(x) - z \rv$ in \cref{priv:ex:exponential} for a function $g: [0,1) \lra \IR$ we obtain the following corollary.
\begin{corollary}[General Laplace perturbation]
	\label{cor:laplace:perturbation}
		Let $g: \mc X \lra \IR$ be a function. Let $g(X)$ be the quantity that we want to anonymize. We define
	\begin{align*}
		\Delta(g) : = \sup_{x, x' \in \mc X} \lv g(x) - g(x') \rv.
	\end{align*}
		Let $\alpha > 0$ and $\Delta(g) < \infty$.  Then,
	\begin{align*}
		Z = g(X) + b \xi, \qquad \xi \sim \text{Laplace}(0,1)
	\end{align*}
	with $b \geq \frac{\Delta(g)}{\alpha}$ is an $\alpha$-differentially private view of $g(X)$ and, hence, of $X$.
\end{corollary}

Let us now come back to the kernel and projection estimators $\tilde{f}_h(t)$ and $\tilde{f}_d(t)$ given in \eqref{KDE} and \eqref{PDE}, respectively. For $t, x \in \IR$, $h > 0$ and $d \in \IN$ we define
\begin{align}
	\label{g}
	g_{h}(x) := K_h(x - t) \qquad \text{ and }  \qquad g_{d}(x) :=  \sum_{j=1}^d \varphi_j(x )\varphi_j(t).
\end{align}
Clearly, the estimators can be written as $\tilde{f}_h(t)= \tfrac{1}{n} \sum_{i=1}^n g_{h}(X_i)$ and $\tilde{f}_d(t) = \frac{1}{n} \sum_{i=1}^n g_{d}(X_i)$. Motivated by \cref{naive} the $i$-th data holder is asked to release a sanitized version  $Z_{i,h}$ of $g_{h}(X_i)$ resp. $Z_{i,d}$ of $g_{d}(X_i)$ rather than privatizing $X_i$ directly. We will apply \cref{cor:laplace:perturbation} to $g = g_{h}$ for kernel density estimation in \cref{sec:KDE} and to $g = g_{d}$ for projection density estimation in \cref{sec:PDE} to verify the privacy constraints. The private kernel and projection estimators are then given by
\begin{align}
	\label{estimator:priv}
	\hat{f}_h(t) = \tfrac{1}{n} \sum_{i=1}^n Z_{i,h}, \qquad \text{ respectively } \qquad \hat{f}_d(t) = \tfrac{1}{n} \sum_{i=1}^n Z_{i,d}.
\end{align}

\paragraph{Private minimax theory.}
Let $\mc E$ be a class of densities and let $(X_i)_{i \in \lcb 1, \dots, n \rcb}$ be (non-private) random variables. Recall that $\mathcal{Q}_\alpha$ denotes the family of all privatization mechanisms satisfying $\alpha$-differential privacy. For $\mathcal{Q} \in \mathcal{Q}_\alpha$ and the true density $f$ we use the symbol $\IE_{f, \mathbbm{Q}}$ for the expectation associated with the distribution $\IP_{f, \mathbbm{Q}}$  of the private view $(Z_i)_{i \in\lcb 1, \dots, n \rcb}$. If there is no confusion, we omit the index.  We derive lower bounds for the private minimax risk given by
\begin{align*}
\inf_{T_n} \inf_{\substack{ \mathbbm{Q} \in \mathcal{Q}_\alpha}} \sup_{f \in \mc P} \IE_{f, \mathbbm{Q}} \lb \lv T_n - f(t) \rv^2  \rb. 
\end{align*}
and a matching upper bound for the estimators in \eqref{estimator:priv} combined with the privacy mechanism from \cref{cor:laplace:perturbation}. Roughly speaking, we are looking for an optimal combination of a privacy mechanism $\privmech$ and an estimation procedure $T_n$ rather than for a minimax-optimal estimator only. Comparing the private minimax risk and the classical minimax risk allows us to characterise the price to pay for data privacy. 

\paragraph{Related literature.}  The first result for a projection approach for estimating a density under privacy constraints is due to \cite{WassermanZhou2010}, Section 6. In a non-local setting they are able to achieve the minimax rate (with sample size replaced by the effective sample size $\alpha^2 n$) using a projection density estimator with Laplace perturbation of the coefficients. In a local setting \cite{DuchiJordanWainwright2018} consider projection density estimation based on privatized views of the observations of the density with respect to $\mc L^2([0,1])$ loss. The non-private minimax estimation risk is well-known to be of order $n^{-\frac{2\sPara}{2\sPara +1}}$, where $\sPara$ is the smoothness parameter of a Sobolev ellipsoid. They show that the local $\alpha$-private minimax risk is of order $(\alpha^2 n)^{-\frac{2 \sPara}{2\sPara + 2}}$, providing both a lower and an upper bound. \cite{ButuceaDuboisKrollSaumard2020} consider Besov ellipsoids with wavelet techniques combined with a Laplace perturbation approach. Also in this case, the privatization causes a deterioration of the order of the risk from $n^{-\frac{2 \beta}{2\beta+1}}$ to $(n(e^\alpha - 1)^2)^{-\frac{2\beta}{2\beta+2}}$ ($\beta$ being the smoothness parameter of the Besov ellipsoid, we only state the \textit{dense zone} here for illustration purposes). Note that this is comparable to the results of  \cite{DuchiJordanWainwright2018} since for small $\alpha$ we have $\alpha \approx e^\alpha - 1$. 
Kernel estimators are, for instance, treated in \cite{HallRinaldoWasserman2013} and \cite{Kroll2019a} under local differential approximate ($\alpha$,$\beta$) - privacy, which is a relaxation of the constraint we consider. These papers consider Laplace perturbation and Gaussian perturbation (which is only useful in the ($\alpha$,$\beta$)-differential privacy context).  \cite{Kroll2021} considers pointwise kernel density estimation under $\alpha$-differential privacy, but considers $\alpha$ to be a constant, thus not making the dependence of the minimax rates on the privacy parameter $\alpha$ explicit. \cite{ButuceaIssartel2021} consider private estimation of nonlinear functionals. 
The results mentioned so far address estimation problems. Concerning testing tasks we mention two recent papers \cite{Lam-WeilLaurentLoubes2022} and  \cite{BerrettButucea2020}.

 \paragraph{Adaptivity under privacy constraints.}
 Also in the private setting the optimal choice of the tuning parameters of the estimators in \eqref{estimator:priv} depends on the regularity of the unknown function $f$.  In the non-private setup a common method for choosing them is the Goldenshluger-Lepski method. Let us heuristically explain the idea behind it in the example of the kernel density estimator $\hat{f}_h(t)$ with bandwidth $h > 0$. Given a classical $\text{bias}^2$-variance-decomposition of the private risk of $\hat{f}_h(t)$ we construct estimators $\hat{R}(h) = V(h) + A(h)$, where $V(h)$ and $A(h)$ approximate the variance term respectively the $\text{bias}^2$-term. We choose the bandwidth that minimizes $\hat{R}(h)$ over a finite collection $\mc H$ of bandwidths. Compared to the non-private setup our proposed method differs in two 
 ways. Firstly, the estimators $V(h)$ and $A(h)$ need to be adjusted to only make use of the available privatized observations. Secondly, we need access to estimators for several tuning parameters. Data holder $i$ is asked to release sanitized version of $g_{h}(X_i)$ for each $h \in \mc H$ with adjusted privacy parameter. The next lemma allows to guarantee privacy for several released observations. For a proof we refer to Lemma 2.16 in \cite{Kroll2019a} (where one sets $\beta_1 = \beta_2 = 0$).
 \begin{lemma}[Composition lemma] \label{composition}
 	Let $Z_1, Z_2$ be $\alpha_1-$ respectively $\alpha_2-$ differentially private views of $X$, which are conditionally on $X$ independent. Then, $Z= (Z_1, Z_2)$ is a $(\alpha_1 + \alpha_2)$-differentially private view of $X$.
 \end{lemma}
Let us briefly point out how to apply the composition lemma in our setting. The $i$-th data holder publishes $Z = (Z_{i,h}, h \in \mc H)$ with adapted privacy parameter $\alpha / \lv \mc H \rv$, where $\lv \mc H \rv$ denotes the number of elements in $\mc H$. Then $Z$ fulfils $\alpha$-differential privacy. 
 
 \paragraph{Related literature.}  \cite{ButuceaDuboisKrollSaumard2020} consider density estimation using wavelets. As usual in the context of wavelet estimation using non-linear thresholding they are able to obtain an adaptive estimator. \cite{Kroll2019a} and \cite{Kroll2021} also address adaptivity issues. In the kernel density estimation case \cite{Kroll2021} suggests a privatised version of Lepski's method, which is different from ours, both in the formulations and the proof techniques and it requires the a priori knowledge of an upper bound for the sup-norm of the true density. Moreover, both the privacy parameter $\alpha$ and the size of the collection of bandwidths $\lv \mc H \rv$ are considered to be constants and do, thus, not appear in the minimax rate. In our study, we make this dependence on $\alpha$ and $\lv \mc H \rv$ explicit. 
\section{Privatised kernel density estimation}
\label{sec:KDE}
\subsection{Upper bound for KDE}
In this section we consider a privatised version of the kernel density
estimator \eqref{KDE}. Note that the estimator is a mean over
evaluated kernel functions. We, therefore, consider the following
privacy mechanism. Let $t \in \IR$ and $h \in (0,1)$ be fixed. For
$i\in\lcb1,\dotsc,n\rcb$ the $i$-th data holder releases
\begin{align}
  \label{eq:privat}
  Z_{i,h} = g_h(X_i) + \frac{2 \lV K \rV_\infty}{\alpha h} \xi_{i,h}, \qquad \xi_{i,h} \iid \mathrm{Laplace}(0,1),
\end{align}
where $g_h(x):=K_h(x-t)$, $x\in\IR$. 
\cref{cor:laplace:perturbation} shows that $Z_{i,h}$ is an $\alpha$-differentially private view of $X_i$, since
\begin{align*}
	\Delta(g_h) = \sup_{x, x' \in [0,1)} \lv K_h(x-t) - K_h(x' -t) \rv \leq \frac{2}{h} \lV K \rV_\infty.
\end{align*}
Based on the private views $Z_{1,h}, \dots, Z_{n,h}$ a natural
estimator of $f_h(t)$ is given  as in \eqref{estimator:priv} by 
	$\hat{f}_h(t) = \tfrac{1}{n} \sum_{i=1}^n Z_{i,h}$.  
Note that due to the centredness of the added Laplace noise, we have
\begin{align*}
	\IE(	\hat{f}_h(t)) = \frac{1}{n} \sum_{i=1}^n \IE Z_{i,h} = \frac{1}{n} \sum_{i=1}^n \IE K_h(X_i - t) = \IE(\tilde{f}_h(t)),
\end{align*}
where 	$\tilde{f}_h$ denotes the non-private estimator defined in \eqref{KDE} and we denote $f_h(t) = \IE(\hat{f}_h(t))$. For simplicity of notation we use the symbol $\IE$ without an index for the expectation over the randomness in the data and in the privatisation.  The following \cref{prop:upper:bound} presents an upper bound on the mean squared error. The upper bound is given by a classical bias$^2$-variance trade-off, which contains the standard variance term of order $\frac{1}{nh}$ and an additional variance term of order $\frac{1}{n \alpha^2 h^2}$. 
	\begin{proposition}[Risk bound]
	\label{prop:upper:bound}
	Consider the privacy mechanism \eqref{eq:privat} and the estimator \eqref{estimator:priv}. For $\alpha \in (0,1)$, $n \in \IN$ and $h \in (0,1)$ we have 
	\begin{align*}
		\IE \lb \lv\hat{f}_h(t) - f(t)\rv^2\rb \leq  \lv f_h(t) - f(t)\rv^2 + \frac{\lV f \rV_\infty \lV K \rV_2^2}{nh} + \frac{8 \lV K \rV_\infty^2}{n \alpha^2 h^2} .
	\end{align*}
\end{proposition}

\begin{proof}
  Recall the notation $f_h(t) = \IE \hat{f}_h(t) $. We have classical
  bias$^2$-variance decomposition of the mean squared error
		\begin{align*}
			\IE\lb \lv\hat{f}_h(t) - f(t)\rv^2\rb = \lv f_h(t)-f(t)\rv^2 + \IE \lb \lv\hat{f}_h(t) - f_h(t)\rv^2\rb.
		\end{align*}
		Let us control the variance term. Recall that we denote $\tilde{f}_h(t) = \frac{1}{n} \sum_{i=1}^n K_h(X_i - t)$. Then due to the independence of the Laplace noise and the original observations, we have
		\begin{align*}
			\var(\hat{f}_h(t)) &= \var(\tilde{f}_h(t)) + \frac{1}{n} \frac{4 \lV K \rV_\infty^2}{\alpha^2 h^2}  \var\lb \xi_{1,h} \rb \\
			& \leq \frac{\lV f \rV_\infty \lV K \rV_2^2}{n h} + \frac{8 \lV K \rV_\infty^2}{n \alpha^2 h^2}, 
		\end{align*}
		which ends the proof.
\end{proof}

	Let us illustrate the result \cref{prop:upper:bound}  for a
        class of densities with smoothness $s > 0$. The optimal
        bandwidth is obtained by balancing the bias$^2$-term and the
        variance terms. Here and subsequently for two real sequence
        $(a_n)_{n\in\Nz}$ and $(b_n)_{n\in\Nz}$ we write $a_n \lesssim 
        b_n$ if there exists an universal numerical constant $C>0$
        such that $a_n\leq C b_n$ for all $n\in\Nz$. If $a_n\lesssim
        b_n$ and $b_n\lesssim a_n$ then we write $a_n\asymp b_n$.
	 \begin{corollary}
	 	\label{cor:kde:illustration}
	 	Consider the following class of densities
\begin{align*}
	\mc E = \lcb f: \IR \lra \IR, f \text{ is a density, } \lV f \rV_\infty \leq c, \sup_{h \in (0,1)} h^{-2\beta} \lb f_h(t)  - f(t) \rb^2 \leq c \rcb 
\end{align*}
	for a universal constant $c > 0$. With $h = h_1^\ast \vee h_2^\ast$, $h_1^\ast \asymp n^{-\frac{1}{2\beta +1}}$, $h_2^\ast \asymp \lb \alpha^2 n \rb^{- \frac{1}{2 \beta + 2}}$ we obtain
\begin{align*}
	\sup_{f \in \mc E} \IE \lb \lv \hat{f}_h(t) - f(t)\rv^2 \rb \lesssim ( c + c \lV K \rV_2^2 +  \lV K \rV_\infty^2) \lcb  n^{-\frac{2\beta}{2\beta+1}} \vee (\alpha^2 n)^{-\frac{2\beta}{2\beta+2}} \rcb.
\end{align*}
	 \end{corollary}
         As expected (see related literature in \cref{sec:review}) there is a twofold deterioration in the rate in the private
         regime. Firstly, the exponent changes from
         $-\frac{2\beta}{2\beta+1}$ to
         $-\frac{2\beta}{2\beta+2}$. Secondly, in the private regime
         the effective sample size is given by $\alpha^2 n$, which
         makes the influence of the privacy parameter explicit.  For
         example the Hölder class of smoothness $\beta > 0$ (we refer
         to Definition 2.3 in \cite{Comte2017}) is covered by \cref{cor:kde:illustration}.

\subsection{Lower bound}
Let $\beta, L > 0$.
The \textbf{Hölder class} $\Sigma(\beta, L)$ on $\IR$ is the set of $l = \lfloor \beta \rfloor$-times differentiable functions $f: \IR \lra \IR$ such that the $l$-th derivative satisfies the Hölder equation
\begin{align*}
\lv f^{(l)}(x) - f^{(l)}(x') \rv \leq L \lv x - x' \rv^{\beta - l}, \qquad \forall x, x' \in \IR.
\end{align*}
We additional introduce the class of Hölder smooth densities
\begin{align*}
	\mc P(\beta ,L) = \lcb f: \IR \lra \IR , f \text{ is a density, } f \in \Sigma(\beta, L) \rcb 
\end{align*}
Proposition 1.2 in \cite{Tsybakov2009} shows that for $f \in \mc P(\beta, L)$ we have $\lb f(t) -f_h(t) \rb^2 \leq c h^{2\beta}$, where $c$ is a constant only dependent on $K$, $L$ and $\beta$. That is, the Hölder class of densities fits into the framework of \cref{cor:kde:illustration}. We provide the following lower bound for privatised pointwise density estimation over a Hölder class. We point out that for small $\alpha$ we have $\alpha \approx \exp(\alpha) - 1$, hence, the lower bound matches the upper bound given in \cref{cor:kde:illustration}. Recall that $\mathcal{Q}_\alpha$ denotes the family of all privatisation mechanisms satisfying $\alpha$-differential privacy. For $\mathcal{Q} \in \mathcal{Q}_\alpha$ and the true density $f$ we use the symbol $\IE_{f, \mathbbm{Q}}$ for the expectation of the privatised data.

\begin{theorem}[Lower bound] \ \\
	\label{thm:lb:hoelder}
	Let $\beta, L > 0$.  For a constant $C >0$ only depending on $\beta$ and $L$ we have
	\begin{align*}
		\liminf_{n \lra \infty} \ \inf_{T_n} \ \inf_{\mathbbm{Q} \in \mathcal{Q}_\alpha}  \ \sup_{f \in \mc P({\beta,L})} \IE_{f, \mathbbm{Q}} \lb \lb n(\exp(\alpha)-1)^2 \rb^\frac{2\beta}{2 \beta + 2} \lv T_n - f(t)  \rv^2 \rb \geq C,
	\end{align*}
	where $\inf_{T_n}$ denotes the infimum over all possible estimators based on the privatised data. 
\end{theorem}

\begin{proof}[Proof of \cref{thm:lb:hoelder}]
We use a general reduction scheme for lower bounds based on two hypotheses (cf. \cite{Tsybakov2009}, Section 2.2).
Note that for $\psi_n^2 = (n (\exp(\alpha)-1)^2)^{\frac{2\beta}{2 \beta + 2}}$  due to Markov's inequality it is sufficient to proof
\begin{align*}
&	 \inf_{T_n} \inf_{\substack{ \mathbbm{Q} \in \mathcal{Q}_\alpha}} \sup_{f \in \mc P(\beta,L)} \IE_{f, \mathbbm{Q}} \lb \psi_n^2 \lv T_n - f(t) \rv^2  \rb \\
& \geq	A^2  \ \inf_{T_n} \  \inf_{\substack{ \mathbbm{Q} \in \mathcal{Q}_\alpha}} \ \max_{f \in \lcb f_{0} , f_{1} \rcb } \IP_{f, \mathbbm{Q}} \lb \lv T_n - f(t) \rv \geq A \psi_n \rb \geq C > 0,
\end{align*}
where $f_{0}$, $f_{1}$ are two hypotheses in $\mc P(\beta, L)$, $A > 0$. Below we construct $f_{0}$, $f_{1}$ such that 
\begin{itemize}
	\item[(i)] $f_{0}, f_1 \in \mc P(\beta, L)$,
	\item[(ii)] $\lb f_{0}(t) - f_{1}(t) \rb^2 \geq 4s^2_n$, $s_n = A \psi_n$,
	\item[(iii)] $ \sup_{\mathbbm{Q} \in \mathcal{Q}_\alpha} \mathrm{KL}(\IP_{f_{1}, \mathbbm{Q}}, \IP_{f_{0}, \mathbbm{Q}}) \leq c < \infty$,
\end{itemize}
for $n$ large enough.
Part(iii) of Theorem 2.2.  in \cite{Tsybakov2009} implies that 
\begin{align*}
	\sup_{f \in \mc P(\beta,L)} \IP_f \lb \lv T_n - f(t) \rv \geq s_n \rb & \geq \max_{j \in \lcb 0,1 \rcb} \IP_{f_{i}}  \lb \lv T_n - f_j(t) \rv \geq s_n \rb \geq \max \lb \tfrac{1}{4} \exp(-c), \tfrac{1 - \sqrt{\tfrac{c}{2}}}{2} \rb.
\end{align*}
Let us now construct $f_{0}$, $f_{1}$ such that (i),(ii) and (iii) are satisfied. Let $\varphi_\sigma$ be the density of a normal distribution with mean 0 and variance $\sigma^2$. We choose the hypotheses
\begin{align*}
	f_{0}(x) := \varphi_\sigma(x), \qquad \text{ and } \qquad f_{1}(x) := \varphi_\sigma(x) + \frac{L}{2} h_n^{\beta} H\lb \frac{x-t}{h_n}\rb, \qquad x \in \IR,
\end{align*}
where $h_n = (n (\exp(\alpha)-1)^2)^{-\frac{1}{2 \beta +2}}$ and $H: \IR \lra [0,\infty)$ satisfies 
	$H \in \Sigma(\beta, \tfrac{1}{2})$  $H(0) > 0$, $\lV H \rV_\infty < \infty, \lV H \rV_1 < \infty$ and $\int H(u) \dif u = 0$. One can e.g. choose $H(x) = K(x) - K(x-1)$ with $K$ as in equation (2.34) of \cite{Tsybakov2009}. Let us check that $f_{0}$, $f_{1}$  satisfy the desired constraints. 
\begin{itemize}
	\item[(i)] \underline{ $f_{0}, f_1 \in \mc P(\beta, L)$ } \\
	Clearly, $f_{0} \in \Sigma(\beta, \frac{L}{2})$ for $\sigma$ large enough and, hence, $f_0 \in \mc P(\beta, L)$. Moreover, for $l = \lfloor \beta \rfloor$ we have $f_{1}^{(l)}(x) = \varphi_\sigma^{(l)}(x) + \frac{L}{2} h_n^{\beta - l} H^{(l)}\lb \frac{x-t}{h_n}\rb$. Thus we obtain for $x, x' \in \IR$
	\begin{align*}
		\lv f_{1}^{(l)}(x) - f_{1}^{(l)}(x') \rv & = \lv \varphi_\sigma^{(l)}(x) - \varphi_\sigma^{(l)}(x') \rv + \frac{L}{2} h_n^{\beta-l} \lv H^{(l)}\lb \frac{x-t}{h_n} \rb - H^{(l)}\lb \frac{x'-t}{h_n} \rb \rv \\
		& \leq \frac{L}{2} \lv x - x' \rv^{\beta-l} + \frac{L}{2} h_n^{\beta - l} \lv \frac{x - x'}{h_n} \rv^{\beta - l} = L \lv x - x' \rv^{\beta - l}.
	\end{align*}
	Furthermore, since $H$ integrates to $0$ and is bounded, $f_1$ is a density for $n$ large enough. 
	\item[(ii)] \underline{$\lb f_{0}(t) - f_{1}(t) \rb^2 \geq 4s^2_n$, $s_n = A \psi_n$} \\
An easy computation shows that 
	\begin{align*}
		\lb f_{0}(t) - f_{1}(t) \rb^2 = \frac{L^2}{2^2} h_n^{2 \beta} H(0) = \frac{H(0) L^2}{4} \lb n \lb \exp(\alpha) - 1 \rb^2 \rb^{-\frac{2 \beta}{2 \beta + 2}} = \frac{A^2}{4} \psi_n^2.
	\end{align*}
	\item[(iii)] \underline{$ \sup_{\mathbbm{Q} \in \mathcal{Q}_\alpha} \mathrm{KL}(\IP_{f_{1}, \mathbbm{Q}}, \IP_{f_{0}, \mathbbm{Q}}) \leq c < \infty$} \\
	By \cite{DuchiJordanWainwright2018}, Theorem 1 and Corollary 3 we have for each $\mathbbm{Q} \in \mathcal{Q}_\alpha$
	\begin{align*}
		\mathrm{KL}(\IP_{f_{1}, \mathbbm{Q}}, \IP_{f_{0}, \mathbbm{Q}}) \leq 4 n \lb \exp(\alpha) - 1 \rb^2 \mathrm{TV}^2(\IP_{f_{0}}, \IP_{f_{1}})),
	\end{align*}
 where  $\IP_{f_{1}}$, $\IP_{f_{0}}$ are the measures corresponding to the non-private model. We calculate
	\begin{align*}
		 & \mathrm{TV}(\IP_{f_{0}}, \IP_{f_{1}}) = \int \lv f_{0}(x) - f_{1}(x) \rv \dif x  = \tfrac{L}{2} h_n^{\beta + 1} \lV H \rV_1 = \tfrac{L}{2} (n \lb \exp(\alpha) - 1\rb^2 )^{1/2} \lV H \rV_1.
	\end{align*} 
 and obtain for any $\mathbbm{Q} \in \mathcal{Q}_\alpha$ that 
		$ \mathrm{KL}(\IP_{f_{0}, \mathbbm{Q}}, \IP_{f_{1}, \mathbbm{Q}}) \leq  L^2 \lV H \rV_1^2  < \infty$.
\end{itemize}
This finishes the proof.
\end{proof}
\subsection{Adaptation for privatised KDE} \label{adapt:kde}
As illustrated in \cref{cor:kde:illustration} the optimal bandwidth
depends on properties of the true density $f$, which are unknown in
practice. We therefore propose an adaptive method for choosing a
bandwidth, which is inspired by the method of
\cite{GoldenshlugerLepski2011}, and show it's optimality. Let $\mc
H\subset(0,1)$ be a finite collection of bandwidths. As explained in
the introduction, we need to adjust the privatisation parameter, since
we are releasing privatised version of $K_h(X_i -t)$ for several
$h$. One way to do the adjustment is to exploit the composition lemma
(\cref{composition}) and replace $\alpha$ in the privatisation
algorithm \eqref{eq:privat} by $\frac{\alpha}{\lv \mc H \rv}$, where $\lv \mc H \rv$ is the number of elements in the collection of bandwidths.
We propose the following bandwidth selection method. We mimic the notation of \cite{Comte2017} and denote $\lcb x \rcb_+ = \max(x,0)$. For $h \in \mc H$ we define
\begin{align*}
	\hat{A}(h) & := \max_{\substack{ \eta \in \mc H \\\eta \leq h}} \lcb \lv \hat{f}_h (t) - \hat{f}_\eta(t) \rv^2 - \lb \hat{V}(h) + \hat{V}(\eta) \rb \rcb_+, \\
	\hat{V}(h) & := \lb 2c_1 \frac{\hat \sigma_h^2}{n} + c_2 \frac{1}{nh} \rb \log n,  \qquad \text{with } \hat{\sigma}_h^2 := \frac{1}{n} \sum_{i=1}^n Z^2_{i,h}, 
\end{align*}
where $c_1 := 600$ and $c_2 := 432$. Here $\hat{A}(h)$ and $\hat{V}(h)$ corresponds to a $\text{bias}^2$- respectively variance-estimate for the estimator $\hat{f}_h(t)$, where the variance term is scaled up by a $\log$-term due to the adaptation. We choose the bandwidth $\hat{h}$ as a minimiser
\begin{align}
	\label{adaptive:h}
	\hat{h} &\in \argmin_{h \in \mc H} \lcb \hat{A}(h) + \hat{V}(h) \rcb.
\end{align} 
In contrast to the (non-private) bandwidth selection procedure presented in \cite{Comte2017}, we have a random variance estimate $\hat{V}(h)$ that does not require an a priori knowledge of an upper bound for $\lV f \rV_\infty$. The classical Goldenshluger-Lepski method also uses a double convolution term as en estimate of the $\bias^2$, which requires knowledge of the entire function $t \mapsto K_h(X_i -t)$ for each $h \in \mc H$. Since this is not available in the private setting, the $\bias^2$ estimate had to be modified.  
The next theorem shows the optimality of our proposed procedure. Let us define the deterministic counterparts of $\hat{A}(h)$ and $\hat{V}(h)$, i.e. a $\text{bias}^2$- and a variance-term
\begin{align*}
	\text{bias}^2(h) & := \sup_{\eta \leq h} \lv f_\eta(t) - f(t) \rv^2 \\
V(h) & :=  \lb c_1 \frac{\sigma_h^2}{n} + c_2 \frac{1}{nh} \rb \log n, \qquad \text{with }\sigma_h^2 := \IE Z_{i,h}^2
\end{align*}
with $c_1$ and $c_2$ as above. Note that (up to constants) $V(h)$ is
bounded by the variance term appearing in
\cref{prop:upper:bound} with an additional $\log$-term that is due to
the adaptation.  The values for the
universal constants $c_1$ and $c_2$, though convenient for deriving the theory,
are far too large in practice. Typically, their values are determined
by means of preliminary
simulations as proposed in \cite{ComteRozenholcTaupin2006}, for example.

\begin{theorem}[Oracle inequality] \label{theorem:adapt:ub}
	Let $\alpha \in (0,1)$ and $\mc H\subset(0,1)$ be a collection of bandwidths satisfying $\lv \mc H \rv \leq n$ such that $nh \geq \max(\log n, 1)$ for all $h \in \mc H$. Let $\hat{h}$ be defined by \eqref{adaptive:h}, then
	\begin{align*}
		\IE	\lb\lv	\hat{f}_{\hat{h}}(t) - f(t)\rv^2\rb  \leq 16 \min_{h \in \mc H} \lcb \mathrm{bias}^2(h) + V(h) \rcb + \frac{C}{n \alpha^2},
	\end{align*}
	where $C$ is a constant only depending on $\lV K \rV_\infty$.
\end{theorem}

Before proving the oracle inequality, let us derive an immediate implication in the context of \cref{cor:kde:illustration}. An easy computation shows the following corollary and we omit its proof.
\begin{corollary}
	Consider the class $\mc E$ as in \cref{cor:kde:illustration}. Let $\mc H := \lcb \frac{1}{k}, k \in \lcb 1, \dots, n \rcb \rcb$ and define $\hat{h}$ by \eqref{adaptive:h}. There exists a constant $c$ only depending on $\lV K \rV_\infty$ such that for all $n \in \IN$ we have
	\begin{align*}
		\sup_{f \in \mc E}  \IE \lb \lv \hat{f}_{\hat{h}}(t) - f(t)\rv^2 \rb \lesssim c \lcb  \lb \frac{n}{\log n} \rb^{-\frac{2\beta}{2\beta+1}} \vee \lb \frac{\alpha^2 n}{ \log n} \rb^{-\frac{2\beta}{2\beta+2}} \rcb.
	\end{align*}
\end{corollary}

\begin{proof}[Proof of \cref{theorem:adapt:ub}]
  Due to the key argument given in \cref{lemma:key:arg} we have for each $h \in \mc H$
  \begin{align*}
    \lv\hat{f}_{\hat h}(t) - f(t)\rv^2 &\leq 16 \text{bias}^2(h) + \frac{4}{3} V(h) + 4 \hat{V}(h) \\ 
	& + 28  \max_{\eta \in \mc H }  \lcb  \lv f_\eta(t) - \hat{f}_\eta(t) \rv^2 - \frac{V(\eta)}{3}  \rcb_+ + 8 \max_{ \eta \in \mc H } \lcb V(\eta) - \hat{V}(\eta) \rcb_+.
	\end{align*}
	The last term is bounded as in \cref{lemma:conc:V}. Furthermore, we have $\IE \hat{V}(h) \leq 2V(h)$, since $\IE \hat{\sigma}_h^2 = \sigma_h^2$. It remains to bound the expectation of the second last term. Note that we have
	\begin{align}
		& \IE  \max_{\eta \in \mc H }  \lcb  \lv f_\eta(t) - \hat{f}_\eta(t) \rv^2 - \frac{V(\eta)}{3}  \rcb_+  \leq \sum_{\eta \in \mc H}  \IE  \lcb  \lb f_{\eta}(t) - \hat{f}_{\eta}(t) \rb^2 - \frac{V(\eta)}{3}  \rcb_+ . \label{eq:sum}
	\end{align}
	For $\eta \in \mc H$ and $i \in \lcb 1, \dotsc, n \rcb$ we write  $S_{i,\eta} :=g_{\eta}(X_i)= K_{\eta}(X_i - t)$ such that 
	\begin{align*}
		\hat{f}_{\eta}(t) - f_{\eta}(t) = \frac{1}{n} \sum_{i=1}^n (S_{i,\eta} - \IE S_{i,\eta}) + b_{\eta} \frac{1}{n} \sum_{i=1}^n \xi_{i,\eta}.
	\end{align*}
	We also write $V(\eta) = V_S(\eta) + b_{\eta}^2 V_\xi(\eta)$ with $V_S(\eta) = \lb c_1 \frac{\IE(S_{i,\eta}^2)}{n} + c_2 \frac{1}{n\eta} \rb \log n$ and $V_\xi= \frac{2 c_1}{n} \log n$. Hence,
	\begin{align*}
		&\IE  \lcb  \lv f_{\eta}(t) - \hat{f}_{\eta}(t) \rv^2 - \frac{V(\eta)}{3}  \rcb_+  \\ 
		&\leq 2 \IE \lcb \lv \tfrac{1}{n} \sum_{i=1}^n (S_{i,\eta} - \IE S_{i,\eta})\rv^2 - \frac{V_S(\eta)}{6} \rcb_+  + 2 b_{\eta}^2 \IE \lcb \lv \tfrac{1}{n} \sum_{i=1}^n \xi_{i,\eta} \rv^2 - \frac{V_\xi}{6} \rcb_+ 
	\end{align*}
	The claim follows from \cref{lemma:conc:S,lemma:conc:xi}
        combined with \eqref{eq:sum}, $\lv \mc H \rv \leq n$ and
        $b_{\eta}^2 \leq \frac{4 \lV K \rV_\infty^2}{\alpha^2} n^4$.
\end{proof}

\begin{lemma}[Key argument] \label{lemma:key:arg}
	For all collections $\mc H \subseteq (0,\infty)$ and all $h \in \mc H$ we have
	\begin{align*}
		\lv \hat{f}_{\hat h}(t) - f(t) \rv^2 &\leq 16 \text{bias}^2(h) + \frac{4}{3} V(h) + 4 \hat{V}(h) \\ 
		& + 28  \max_{ \eta\in \mc H}  \lcb  \lv f_{\eta}(t) - \hat{f}_{\eta}(t) \rv^2 - \frac{V(\eta)}{3}  \rcb_+ + 8 \max_{\eta \in \mc H } \lcb V(\eta) - \hat{V}(\eta) \rcb_+.
	\end{align*}
\end{lemma}
\begin{proof}[Proof of \cref{lemma:key:arg}]
	We have $\lv\hat{f}_{\hat{h}}(t) - f(t)\rv^2 \leq 2 \lv \hat{f}_{\hat{h}}(t) - \hat{f}_h(t)\rv^2 + 2 \lv \hat{f}_h(t) - f(t) \rv^2$ for all $h \in \mc H$. Furthermore, for $h \in \mc H$ we obtain
	\begin{align}
		\lv \hat{f}_{\hat{h}}(t) - \hat{f}_h(t) \rv^2 & = 		\lv \hat{f}_{\hat{h}}(t) - \hat{f}_h(t) \rv^2 - \hat{V}(\hat{h} \wedge h) - \hat{V}(\hat{h} \vee h) + \hat{V}(\hat{h} \wedge h) + \hat{V}(\hat{h} \vee h) \nonumber \\
		& \leq \hat{A}(\hat{h} \vee h) + \hat{V}(\hat{h} \wedge h) + \hat{V}(\hat{h} \vee h) \nonumber \\ 
		& \leq \hat{A}(\hat{h} \vee h) + \hat{V}(\hat{h} \vee h) + \hat{A}(\hat{h} \wedge h) + \hat{V}(\hat{h} \wedge h) \nonumber \\
		& \leq 2 \lb \hat{A}(h) + \hat{V}(h) \rb, \label{key:1}
	\end{align}
	where we exploited the definition of $\hat A$, the positivity of both $\hat A$ and $\hat V$ and the minimising property of $\hat{h}$. The term $\hat A(h)$ can further be bounded by 
	\begin{align}
		\hat A(h) & \leq \max_{\substack{ \eta \in \mc H \\\eta \leq h}}  \lcb 3 \lv \hat{f}_h(t)  - f_h(t)\rv^2 + 3 \lv f_h(t) - f_\eta(t)\rv^2 + 3 \lv f_\eta(t) - \hat{f}_\eta(t) \rv^2 - \hat{V}(\eta) - \hat{V}(h) \rcb_+ \nonumber \\
		& \leq \max_{\substack{ \eta \in \mc H \\\eta \leq h}}  \lcb 3 \lv \hat{f}_h(t)  - f_h(t)\rv^2  + 3 \lv f_\eta(t) - \hat{f}_\eta(t) \rv^2 - V(\eta) - V(h) \rcb_+ \nonumber \\
		& \qquad +  3 \max_{\substack{ \eta \in \mc H \\\eta \leq h}} \lv f_h(t) - f_\eta(t)\rv^2 + \max_{\substack{ \eta \in \mc H \\\eta \leq h}} \lcb V(\eta) - \hat{V}(\eta) \rcb_+ + \lcb V(h) - \hat{V}(h) \rcb_+ \nonumber \\
		& \leq 3 \max_{\substack{ \eta \in \mc H \\\eta \leq h}} \lcb  \lv f_\eta(t) - \hat{f}_\eta(t) \rv^2 - \frac{V(\eta)}{3}  \rcb_+  + 3 \lcb  \lv f_h(t) - \hat{f}_h(t) \rv^2 - \frac{V(h)}{3}  \rcb_+ + 3 \text{bias}^2(h)   \nonumber \\
				& \qquad + 2\max_{\substack{ \eta \in \mc H \\\eta \leq h}} \lcb V(\eta) - \hat{V}(\eta) \rcb_+\nonumber \\
		& \leq 6  \max_{\eta \in \mc H}  \lcb  \lv f_\eta(t) - \hat{f}_\eta(t) \rv^2 - \frac{V(\eta)}{3}  \rcb_+ + 3 \text{bias}^2(h) + 2 \max_{ \eta \in \mc H } \lcb V(\eta) - \hat{V}(\eta) \rcb_+.  \label{key:2}
	\end{align}	
	Moreover, we bound
	\begin{align}
		\lv\hat{f}_h(t) - f(t) \rv^2 & \leq 2 \lv \hat{f}_h(t) - f_h(t) \rv^2 + 2 \lv f_h(t) - f(t) \rv^2 \nonumber \\
		& \leq 2 \lcb \lv \hat{f}_h(t) - f_h(t) \rv^2 - \frac{V(h)}{3} \rcb_+ + \frac{2 V(h)}{3} + 2 \text{bias}^2(h). \label{key:3}
	\end{align}
	Combining the bounds \eqref{key:1}, \eqref{key:2} and \eqref{key:3} yields the desired result.	
\end{proof}

\begin{lemma}[Concentration of $\hat{V}$]
	We have
	\label{lemma:conc:V}
	\begin{align*}
		\IE \max_{\eta \in \mc H} \lcb V(\eta) - \hat{V}(\eta) \rcb_+  \leq  \frac{C}{n \alpha^2},
	\end{align*}
where $C$ is a constant only depending on $\lV K \rV_\infty$. 
\end{lemma}
\begin{proof}
  We bound the maximum by the sum,
	\begin{align*}
          \IE	 \max_{\eta \in \mc H } \lcb V(\eta) - \hat{V}(\eta) \rcb_+ \leq \sum_{\eta \in \mc H} \IE \lcb V(\eta) - \hat{V}(\eta) \rcb_+ = c_1 \frac{\log n}{n} \sum_{\eta\in \mc H} \IE \lcb \sigma^2_{\eta} - 2 \hat{\sigma}^2_{\eta} \rcb_+.
	\end{align*}
Consider the event $\Omega_{\eta} := \lcb \lv \sigma^2_{\eta} - \hat{\sigma}^2_{\eta} \rv \leq \tfrac{\sigma^2_{\eta}}{2} \rcb$ and its complement $\Omega_{\eta}^c$. On $\Omega_{\eta}$ we have $\sigma^2_{\eta} \leq 2 \hat{\sigma}^2_{\eta}$ and thus,
\begin{align*}
	\lcb \sigma^2_{\eta} - 2 \hat{\sigma}^2_{\eta} \rcb_+ = 	\lcb \sigma^2_{\eta} - 2 \hat{\sigma}^2_{\eta} \rcb_+ \mathds{1}_{\Omega_{\eta}^c} \leq 2 \sigma^2_{\eta} \lv \frac{\hat{\sigma}^2_{\eta}}{\sigma^2_{\eta}} - 1 \rv \mathds{1}_{\Omega_{\eta}^c}.
\end{align*}
Let $m \in \IN$, to be chosen below. On $\Omega_{\eta}^c$ we have $\lv \tfrac{\hat{\sigma}^2_{\eta}}{\sigma^2_{\eta}} - 1 \rv > \tfrac{1}{2}$, hence we obtain
\begin{align}
	\label{eq:max}
		\IE	 \max_{\eta \in \mc H} \lcb V(\eta) - \hat{V}(\eta) \rcb_+ \leq 2 c_1 \frac{\log n}{n} \sum_{\eta\in \mc H} \sigma^2_{\eta} \frac{\IE \lb  \lv \tfrac{\hat{\sigma}^2_{\eta} }{\sigma^2_{\eta}} - 1 \rv^{m} \rb }{2^{m-1}}.
\end{align}
Let us find an upper bound for $\IE \lb  \lv \tfrac{\hat{\sigma}^2_{\eta} }{\sigma^2_{\eta}} - 1 \rv^{m} \rb$. Note that with $U_{i,\eta} = \frac{Z_{i,\eta}^2}{\sigma^2_{\eta}} - 1$ we have $\tfrac{\hat{\sigma}^2_{\eta} }{\sigma^2_{\eta}} - 1  = \frac{1}{n} \sum_{i=1}^n U_{i,\eta}$, where $U_{i,\eta}$, $i \in \lcb 1 ,\dotsc, n \rcb$ are independent random variables with $\IE U_{i,\eta} = 0$. We apply Petrov's inequality (\cref{petrovs:ineq}), which shows that for $m \geq 2$
\begin{align}
	\label{petrov}
	\IE \lb  \lv \tfrac{\hat{\sigma}^2_{\eta} }{\sigma^2_{\eta}} - 1 \rv^{m} \rb = \IE  \lv \frac{1}{n} \sum_{i=1}^n U_{i,\eta} \rv^{m} \leq c_m n^{-m/2 - 1} \sum_{i=1}^n \IE \lv U_{i,\eta} \rv^m
\end{align}
for a constant $c_m$ only depending on $m$. It remains to find an upper bound for $\IE \lv U_{i,\eta} \rv^m$. First note that due to Jensen's inequality $\IE \lv U_{i,\eta} \rv^m = \IE \lv \frac{Z_{i,\eta}}{\sigma^2_{\eta}} - 1 \rv^m \leq 2^m \lb 1 + \IE \lv \tfrac{Z_{i,\eta}}{\sigma^2_{\eta}} \rv^{m} \rb$. Moreover,
\begin{align}
	\label{eq:z:sigma}
	\IE \lv \tfrac{Z_{i,\eta}}{\sigma^2_{\eta}} \rv^{m} \leq 2^m \frac{ \IE S_{1,\eta}^{2m} + \IE T_{1,\eta}^{2m} }{\sigma_{\eta}^{2m}}.
\end{align}
Since for $X \sim \text{Laplace}(0,1)$ we have $\IE \lv X \rv^m = m!$ (cp. \cref{priv:lap:reminder}), we obtain
\begin{align*}
	\IE T_{1,\eta}^{2m} = b_{\eta}^{2m} (2m)! \qquad \text{ and } \qquad \IE S_{1,\eta}^{2m} = \IE K_{\eta}^{2m}(X_i -t) \leq \frac{1}{\eta^{2m}} \lV K \rV_\infty^{2m}.
\end{align*}
Inserting these bounds into \eqref{eq:z:sigma} and exploiting $\sigma^2_{\eta} / b_{\eta}^2 \geq 2$ yields
\begin{align*}
		\IE \lv \tfrac{Z_{i,\eta}}{\sigma^2_{\eta}} \rv^{m} \leq 2^m \frac{  \lV K \rV_\infty^{2m} \eta^{-2m} + b_{\eta}^{2m} (2m)! }{\sigma_{\eta}^{2m}} \leq  \lV K \rV_\infty^{2m} \eta^{-2m} b_{\eta}^{-2m} + (2m)!
\end{align*}
Recall that $b_{\eta}^2 = \frac{4 \lv \mc H \rv^2 \lV K \rV_\infty^2}{\alpha^2 \eta^2}$ and that $\alpha/ \lv\mc H \rv \leq 1$. Therefore, we obtain
\begin{align*}
		\IE \lv \tfrac{Z_{i,\eta}}{\sigma^2_{\eta}} \rv^{m} \leq  4^{-2m} + (2m)!
\end{align*}
and with \eqref{petrov} finally
\begin{align*}
		\IE \lb  \lv \tfrac{\hat{\sigma}^2_{\eta} }{\sigma^2_{\eta}} - 1 \rv^{m} \rb  \leq c_m n^{m/2 - 1} \sum_{i=1}^n 2^m \lb 1 + 4^{-2m} + (2m)! \rb \leq C_m n^{-m/2},
\end{align*}
where $C_m$ is a constant only dependent on $m$. Inserting this bound into \eqref{eq:max} we obtain
\begin{align*}
		\IE	 \max_{\eta \in \mc H } \lcb V(\eta) - \hat{V}(\eta) \rcb_+ \leq 2 c_1 \frac{\log n}{n} \tilde{C}_m  \sum_{\eta \in \mc H} \frac{\sigma^2_{\eta}}{n^{m/2}} 
\end{align*}
Since $n\eta \geq 1$ and $\lv \mc H \rv \leq n$ we have
\begin{align*}
	\sigma^2_{\eta} \leq \frac{1}{\eta^2} \lV K \rV_\infty^2 + 8 \frac{\lv \mc H \rv^2}{\alpha^2 \eta^2} \lV K \rV_\infty^2 \leq n^2 \lV K \rV_\infty^2 + 8 n^4 \lV K \rV_\infty^2 \frac{1}{\alpha^2},
\end{align*}
i.e. $n^{-4} \sigma^2_{\eta} \leq \lV K \rV_\infty \lb 1 + \frac{8}{\alpha^2} \rb$. Hence with $m = 12$ we obtain the claim. 
\end{proof}

\begin{lemma}[Concentration for $\sum_{i=1}^nS_{i,\eta}$]
	\label{lemma:conc:S} \text{} \\
	Recall the definition $V_S(\eta) = \lb c_1 \frac{\IE (S_{i,\eta}^2)}{n} + c_2 \frac{1}{n\eta} \rb \log n$. We have
	\begin{align*}
		\IE \lcb \lv \frac{1}{n} \sum_{i=1}^n (S_{i,\eta} - \IE S_{i,\eta})\rv^2 - \frac{V_S(\eta)}{6} \rcb_+ \leq 128 \lV K \rV_\infty \frac{1}{n^2}.
	\end{align*}
\end{lemma}
\begin{proof}[Proof of \cref{lemma:conc:S}]
	We have 
	\begin{align}
		 \IE \lcb \lv \frac{1}{n} \sum_{i=1}^n (S_{i,\eta} - \IE S_{i,\eta})\rv^2 - \frac{V_S(\eta)}{6} \rcb_+ \nonumber \\ = \int_{0}^\infty \IP \lb  \lv   \frac{1}{n} \sum_{i=1}^n (S_{i,\eta} - \IE S_{i,\eta}) \rv \geq \sqrt{\frac{V_S(\eta)}{6}  + x} \rb \dif x. \label{integral}
	\end{align}
	We note that $S_{i,\eta}$, ${i \in \lcb 1, \dotsc, n \rcb}$, are iid with $\lv S_{1,\eta} \rv \leq \tfrac{\lV K \rV_\infty}{\eta} =: \mathrm{b}$ and $\var(S_{1,\eta}) \leq \IE K_{\eta}^2(X_1 -t) =: \mathrm{v}^2$. We can write $V_S(\eta) = \lb c_1 \frac{\mathrm{v}^2}{n} + \frac{c_2}{\lV K \rV_\infty} \frac{\mathrm{b}}{n} \rb \log n$. Consequently, the Bernstein type inequality 	\cref{bernstein:ineq} yields for $\eps = \sqrt{\frac{V_S(\eta)}{6} + x}$ the following
	\begin{align*}
		& \IP \lb \lv \frac{1}{n} \sum_{i=1}^n (S_{i,\eta} - \IE S_{i,\eta}) \rv \geq  \sqrt{\frac{V_S(\eta)}{6} + x}  \rb \\
		& \leq 2 \max \lcb \exp\lb - \frac{n}{4 \mathrm v^2} \lb \frac{V_S(\eta)}{6} + x  \rb \rb, \exp \lb - \frac{n}{4 \mathrm b} \sqrt{ \frac{V_S(\eta)}{6} + x} \rb \rcb \\
		& \leq 2 \exp(-3 \log n) \max \lcb \exp\lb - \frac{n}{4 \mathrm v^2} x  \rb, \exp \lb - \frac{n}{4 \sqrt{2} \mathrm b} \sqrt{ x} \rb \rcb,
	\end{align*}
	since  $\sqrt{2} \sqrt{x+y} \geq \sqrt{x} + \sqrt{y}$ for $x,y \geq 0$ and $n \eta \geq \log n$. 
	Finally, we obtain
	\begin{align}
		& \IP \lb \lv \frac{1}{n} \sum_{i=1}^n (S_{i,\eta} - \IE S_{i,\eta}) \rv \geq  \sqrt{\frac{V_S(\eta)}{6} + x}  \rb
		& \leq \frac{2}{n^3} \max \lcb \exp(-\tau_1 x), \exp(-\tau_2 \sqrt{x})\rcb. \label{prop:bound}
	\end{align}
	with $\tau_1 = \frac{n}{4 \mathrm{v}^2}$ and $\tau_2 =\frac{n}{4 \sqrt{2} \mathrm{b}} $. 
	Note that we have $\int_{0}^\infty \exp(-\tau_1 x) \dif x = \tfrac{1}{\tau_1} \leq 4 \lV K \rV_\infty^2 n$ and $\int_{0}^\infty \exp(-\tau_2 \sqrt{x}) \dif x = \frac{2}{\tau^2} \leq 4^3 \lV K \rV_\infty^2 n$, where we again exploited that $n \eta \geq 1$. Hence, inserting the bound \eqref{prop:bound} into \eqref{integral} and integrating yields the result. 
\end{proof}

\begin{lemma}[Concentration for $\sum_{i=1}^n\xi_{i,\eta}$] \label{lemma:conc:xi}
	Recall the definition $V_\xi =  \frac{2 c_1}{n} \log n$. We have
	\begin{align*}
		\IE \lcb \lv \frac{1}{n} \sum_{i=1}^n \xi_{i,\eta} \rv^2 - \frac{V_\xi}{6} \rcb_+  \leq \frac{32}{n^6}.
	\end{align*}
\end{lemma}
\begin{proof}[Proof of \cref{lemma:conc:xi}]
	We have 
	\begin{align}
		& \IE \lcb \lv \frac{1}{n} \sum_{i=1}^n \xi_{i,\eta} \rv^2 - \frac{V_\xi}{6} \rcb_+ = \int_{0}^\infty \IP \lb  \lv  \frac{1}{n} \sum_{i=1}^n \xi_{i,\eta} \rv \geq \sqrt{\frac{V_\xi}{6}  + x} \rb \dif x. \label{integral:2}
	\end{align}
	Consequently, the tail bound \cref{prop:laplace} for $\eps = \sqrt{\frac{V_\xi}{6} + x}$ with $V_\xi =  \frac{2c_1}{n} \log n $ yields
	\begin{align*}
		&\IP\lb \lv \frac{1}{n} \sum_{i=1}^n \xi_{i,\eta} \rv \geq \sqrt{\frac{V_\xi}{6} + x}  \rb  \leq 2 \max \lcb \exp\lb - \frac{n \eps^2}{16 }\rb, \exp\lb - \frac{n \eps}{2} \rb \rcb \\
		& \leq 2 \max \lcb \exp\lb - \frac{n}{16} \lb \frac{V_\xi}{6} + x  \rb \rb, \exp \lb - \frac{n}{2} \sqrt{ \frac{V_\xi}{6} + x} \rb \rcb \\
		& \leq 2 \exp(-5 \log n) \max \lcb \exp\lb - \frac{n}{16} x  \rb, \exp \lb - \frac{n}{2 \sqrt{2}} \sqrt{ x} \rb \rcb,
	\end{align*}
	where we used $\sqrt{2} \sqrt{x+y} \geq \sqrt{x} + \sqrt{y}$ for $x,y \geq 0$.
	Hence, for $\tau_1 =\frac{n}{16}$ and $\tau_2 =  \frac{n}{2 \sqrt{2}}$ we obtain
	\begin{align}
		& \IP \lb \lv \frac{1}{n} \sum_{i=1}^n \xi_{i,\eta} \rv \geq  \sqrt{\frac{V_\xi}{6} + x}  \rb  \leq \frac{2}{n^5} \max \lcb \exp(-\tau_1 x), \exp(-\tau_2 \sqrt{x})\rcb. \label{prop:bound:2}
	\end{align}
	Note that we have $\int_{0}^\infty \exp(-\tau_1 x) \dif x = \tfrac{1}{\tau_1} = \frac{16}{n}$ and $\int_{0}^\infty \exp(-\tau_2 \sqrt{x}) \dif x = \frac{2}{\tau_2^2} = \frac{8}{n}$. Thus, inserting the bound \eqref{prop:bound:2} into \eqref{integral:2} and integrating yields the result. 
\end{proof}

\section{Privatized projection density estimation}\label{sec:PDE}
\subsection{Upper bound for PDE}
In this section we consider a privatised version of the projection
density estimator \eqref{PDE} where we assume the true density to
belong to $\mc L^2([0,1])$.  Here and subsequently, let
$\{\varphi_j\}_{j\in\Nz}$ denote an arbitrary orthonormal system of
$\mc L^2([0,1])$ satisfying the following assumption.
\begin{assumption}\label{ass:basis}There is a finite constant $\phi_o$ such that
  $\lV\sum_{j=1}^d\varphi_j^2\rV_\infty\leq d\phi_o$ for all $d\in\Nz$. 
\end{assumption}
According to Lemma 6 of \cite{BirgeMassart1997} our  Assumption \ref{ass:basis} is
exactly equivalent to following property: there exists a positive constant $\phi_o$ such that for any
  $h\in\text{lin}\big((\varphi_i)_{i\in\{1,\dotsc,d\}}\big)$ we have
  $\lV h\rV_\infty^2\leq \phi_od \lV h\rV^2_{2}$.
Typical examples are bounded basis, such as the trigonometric
basis, or bases satisfying the condition:  there exists a positive constant $C$ such that for any
$c\in\Rz^{d}$,  $\VnormInf{\sum_{i=1}^{d} c_i\varphi_i}^2\leq
C d|c|_\infty^2$ where
$|c|_{\infty}=\max\{|c_i|\}_{i\in\{1,\dotsc,d\}}$.
\cite{BirgeMassart1997}  show that the last property is satisfied
for piecewise polynomials, splines and wavelets.

Note that the estimator is a mean over
evaluated basis functions. We, therefore, consider the following
privacy mechanism. Let $t \in [0,1]$ and $d \in \Nz$ be fixed. For
each $i\in\{1,\dotsc,n\}$ the $i$-th data holder releases
\begin{align}
  \label{pde:eq:privat}
	Z_{i,d} = g_{d}(X_i) + \frac{2 \phi_o d}{\alpha} \xi_{i,d}, \qquad \xi_{i,d} \iid \text{Laplace}(0,1),
\end{align}
writing $g_{d}(x):=\sum_{j=1}^d\varphi_j(x)\varphi_j(t)$, $x\in[0,1]$. \cref{cor:laplace:perturbation} shows that $Z_{i,d}$
is an $\alpha$-differentially private view of $X_i$, since
\begin{align*}
\Delta(g_{d}) = \sup_{x, y \in [0,1]} \lv g_{d}(x)-g_{d}(y)  \rv \leq \sup_{x, y \in [0,1]}
  \sum_{j=1}^d\lv\varphi_j(x)-\varphi_j(y)\rv\,\lv\varphi_j(t)\rv\leq
  2\lV\sum_{j=1}^d\varphi_j^2\rV_\infty\leq 2\phi_o d.
\end{align*}
Based on the private views $(Z_{i,d})_{i\in\{1,\dotsc,n\}}$ a natural
estimator is given as in \eqref{estimator:priv} by $\hat{f}_d(t) = \frac{1}{n} \sum_{i=1}^n Z_{i,d}$.  
Note that due to the centredness of the added Laplace noise, we have
\begin{align*}
	\IE(	\hat{f}_d(t)) = \frac{1}{n} \sum_{i=1}^n \IE Z_{i,d} = \frac{1}{n} \sum_{i=1}^n \IE g_{d}(X_i) = \IE(\tilde{f}_d(t)),
\end{align*}
where $\tilde{f}_d$ denotes the non-private estimator defined in
\eqref{PDE}. Let us introduce
$f_d(t) = \IE(\hat{f}_d(t)) = \sum_{j=1}^d\la f,\varphi_j\ra\varphi_j(t)$. The following \cref{pde:prop:upper:bound} gives an
upper bound on the mean squared error. The upper bound is given by a
classical bias$^2$-variance trade-off, which contains the standard
variance term of order $\tfrac{d}{n}$ and an additional variance term
of order $\tfrac{d^2}{n\alpha^{2}}$.
\begin{proposition}[Risk bound]
  \label{pde:prop:upper:bound}
  Consider the privacy mechanism \eqref{pde:eq:privat} and the
  estimator \eqref{estimator:priv}. For $\alpha \in (0,1)$ and $n,d \in \IN$  we have  
  \begin{align*}
    \IE \big(|\hat{f}_d(t) - f(t)|^2\big) \leq  | f_d(t) - f(t)|^2 + \frac{\lV f \rV_\infty \phi_od}{n} + \frac{8 \phi_o^2d^2}{n \alpha^2}.
  \end{align*}
\end{proposition}
\begin{proof}
  Recall the notation $f_d(t) = \IE \hat{f}_d(t) $. We have classical bias$^2$-variance decomposition of the mean squared error
  \begin{align*}
    \IE\big(|\hat{f}_d(t) - f(t)|^2\big) = |f_d(t)-f(t)|^2 + \IV( \hat{f}_d(t)).
  \end{align*}
Let us control the variance term.  Due to the independence of the Laplace noise and the original observations, we have
\begin{equation*}
  n\IV(\hat{f}_d(t)) = \IV\big(g_{d}(X_1)\big) + \frac{4 \phi_o^2 d^2}{\alpha^2}  \IV(\xi_{1,d}) 
  \leq \lV f \rV_\infty \phi_o d+ \frac{8\phi_o^2 d^2}{\alpha^2 n} ,
\end{equation*}
which ends the proof.
\end{proof}
Let us illustrate the result \cref{pde:prop:upper:bound} for a class
of densities with smoothness $\beta > 0$. The optimal dimension is
obtained by balancing the bias$^2$-term and the variance terms.
\begin{corollary}
  \label{cor:pde:illustration}
  Consider the following class of densities
  \begin{equation*}
	\mc E = \lcb f\in\mc L^2([0,1]) : f \text{ is a density }, \lV
        f \rV_\infty \leq c,\,\sup\{d^{2s}|f_d(t)  - f(t)|^2:d\in\Nz\} \leq c \rcb 
\end{equation*}
for a universal constant $c > 0$. With $d = d_1^\ast \wedge d_2^\ast$,
$d_1^\ast \asymp n^{\frac{1}{2s+1}}$, $d_2^\ast \asymp \lb \alpha^2 n
\rb^{\frac{1}{2s + 2}}$ we obtain 
\begin{align*}
	\sup_{f \in \mc E} \IE \big(|\hat{f}_d(t) - f(t)|^2\big)
  \lesssim (c+c\phi_o+\phi_o^2)\; \lcb n^{-\frac{2\beta}{2\beta+1}} \vee (\alpha^2 n)^{-\frac{2\beta}{2\beta+2}}\rcb.
\end{align*}
\end{corollary}
We observe the same twofold deterioation in the rate in the private
regime as in kernel density estiamtion. As an example for a class of densities that is covered by \cref{cor:pde:illustration} we refer to Definition 1.11. in \cite{Tsybakov2009} (Sobolev spaces). 
\subsection{Lower bound}

Let $\beta \geq 1$ be an integer and $L > 0$. The \textbf{Sobolev class} $\mathcal{W}(\beta,L)$ is the set of all $(\beta - 1)$-times differentiable functions $f: [0,1] \lra \IR$ such that the $(\beta-1)$-th derivative is absolutely continuous and satisfies $\int \lb f^{(\beta)}(x) \rb^2 \dif x \leq L^2$ (cp. Section 1.7.1 in \cite{Tsybakov2009}). We denote by $\mc P_{\mc W}(\beta,L)$ the set of densities in $\mathcal{W}(\beta, L)$.

\begin{theorem}[Lower bound] \ \\
	\label{thm:lb:sobolev}
	Let $\beta, L > 0$.  For a constant $C >0$ only depending on $\beta$ and $L$ we have
	\begin{align*}
		\liminf_{n \lra \infty} \ \inf_{T_n} \ \inf_{\mathbbm{Q} \in \mathcal{Q}_\alpha}  \ \sup_{f \in \mc P_{\mc W}({\beta,L})} \IE_{f, \mathbbm{Q}} \lb \lb n(\exp(\alpha)-1)^2 \rb^\frac{2\beta}{2 \beta + 2} \lv T_n - f(t)  \rv^2 \rb \geq C,
	\end{align*}
	where $\inf_{T_n}$ denotes the infimum over all possible estimators based on the privatised data. 
\end{theorem}
\begin{proof}[Proof of \cref{thm:lb:sobolev}]
	It is easily verified that the Hölder class $\Sigma(\beta,L)$ restricted to $[0,1]$ is contained in the Sobolev class $\mc W(\beta,L)$ (cf. also Section 1.7.1 in \cite{Tsybakov2009}). Hence, the densities constructed in the reduction scheme in the proof of the lower bound of \cref{thm:lb:sobolev} are contained in $\mc P_{\mc W}(\beta,L)$, if we replace the normal density $\varphi_\sigma$ by the uniform density to guarantee that their support is contained in $[0,1]$. The proof then follows line by line the proof of \cref{thm:lb:sobolev}.
\end{proof}

\subsection{Adaptation for privatised PDE}
Similar to \cref{adapt:kde} we propose the following dimension selection method. Let $\mc D \subseteq \IN$ be a finite collection of dimension. As before we adjust the privacy parameter $\alpha$ in \eqref{pde:eq:privat} to $\tfrac{\alpha}{ \lv \mc D \rv}$ and for $d \in \mc D$ define 
\begin{align*}
	\hat{A}(d) & := \max_{\substack{ D \in \mc D \\D\geq d}} \lcb \lv \hat{f}_d (t) - \hat{f}_D(t) \rv^2 - \lb \hat{V}(d) + \hat{V}(D) \rb \rcb_+, \\
	\hat{V}(d) & := \lb 2c_1 \frac{\hat \sigma_d^2}{n} + c_2 \frac{d}{n} \rb \log n,  \qquad \text{with } \hat{\sigma}_d^2 := \frac{1}{n} \sum_{i=1}^n Z^2_{i,d}, 
\end{align*}
where $c_1 := 600$ and $c_2 := 432$ as in \cref{adapt:kde}. We choose the dimension $\hat{d}$ as a minimiser
\begin{align}
	\label{pde:adaptive:d}
	\hat{d} &\in \argmin_{d \in \mc D} \lcb \hat{A}(d) + \hat{V}(d) \rcb.
\end{align} 
Let us define the deterministic counterpart of  ${V}(d)$, i.e. a $\text{bias}^2$- and a variance-term
\begin{align*}
	\text{bias}^2(d) & := \sup_{D \geq d} |f_D(t) - f(t)|^2 \\
V(d) & :=  \lb c_1 \frac{\sigma_d^2}{n} + c_2 \frac{d}{n} \rb \log n, \qquad \text{with }\sigma_d^2 := \IE Z_{1,d}^2
\end{align*}
with $c_1$ and $c_2$ as above. 
We see at once that identifying $\frac{1}{d}$ with $h$ in the results of \cref{adapt:kde}, the following oracle inequality is proven by exactly the same methods as \cref{theorem:adapt:ub}. For the convenience of the reader we repeat the relevant steps of the proof, thus making our exposition complete.

\begin{theorem}[Oracle inequality] \label{pde:theorem:adapt:ub}
	Let $\alpha \in (0,1)$ and $\mc D\subset\Nz$ be a collection of
        dimension parameters satisfying $\lv \mc D \rv \leq n$ such that $n/d
        \geq \max(\log n, 1)$ for all $d \in \mc D$. Let $\hat{d}$ be
        defined by \eqref{pde:adaptive:d}, then there is a  constant $C$ depending only on $\phi_o$ such that 
	\begin{align*}
		\IE		\lb \lv \hat{f}_{\hat{d}}(t) - f(t) \rv^2 \rb \leq
          16 \min_{d \in \mc D} \lcb \mathrm{bias}^2(d) + V(d) \rcb +
          \frac{C \phi_o^2 }{n \alpha^2}.
	\end{align*}
\end{theorem}

Similarly to \cref{adapt:kde} before proving the oracle inequality, we first provide an immediate implication in the context of \cref{cor:pde:illustration} and omit its proof.

\begin{corollary}
	Consider the class $\mc E$ as in \cref{cor:pde:illustration}. Let $\mc D := \lcb 1, \dots, n \rcb$ and define $\hat{d}$ by \eqref{pde:adaptive:d}. There exists a constant $c$ only depending on $\phi_o$ such that for all $n \in \IN$ we have
	\begin{align*}
		\sup_{f \in \mc E}  \IE \lb \lv \hat{f}_{\hat{d}}(t) - f(t)\rv^2 \rb \lesssim c \lcb  \lb \frac{n}{\log n} \rb^{-\frac{2\beta}{2\beta+1}} \vee \lb \frac{\alpha^2 n}{ \log n} \rb^{-\frac{2\beta}{2\beta+2}} \rcb.
	\end{align*}
\end{corollary}

\begin{proof}[Proof of \cref{pde:theorem:adapt:ub}]
Identifying $h=d^{-1}$ the key argument given in
\cref{lemma:key:arg} for each $d\in \mc D$ implies
		\begin{align*}
		\lv \hat{f}_{\hat d}(t) - f(t) \rv^2 &\leq 16 \text{bias}^2(d) + \frac{4}{3} V(d) + 4 \hat{V}(d) \\ 
	& + 28  \max_{D \in \mc D}  \lcb  \lv f_D(t) - \hat{f}_D(t) \rv^2 - \frac{V(D)}{3}  \rcb_+ + 8 \max_{D \in \mc D} \lcb V(D) - \hat{V}(D) \rcb_+.
	\end{align*}
	The last term is bounded as in \cref{pde:lemma:conc:V}. Furthermore, we have $\IE \hat{V}(d) \leq 2V(d)$, since $\IE \hat{\sigma}_d^2 = \sigma_d^2$. It remains to bound the expectation of the second last term. Note that we have
	\begin{align}
		& \IE  \max_{ D \in \mc D}  \lcb  \lv f_D(t) - \hat{f}_D(t) \rv^2 - \frac{V(D)}{3}  \rcb_+  \leq \sum_{D \in \mc D}  \IE  \lcb  \lv f_D(t) - \hat{f}_D(t) \rv^2 - \frac{V(D)}{3}  \rcb_+ . \label{pde:eq:sum}
	\end{align}
	For $d \in \mc D$ and $i \in \lcb 1, \dots, n \rcb$ we write
        $S_{i,d} := g_{d}(X_i)$ and $b_d:= 2\phi_o|\mc D| d/\alpha$ such that 
	\begin{align*}
		\hat{f}_d(t) - f_d(t) = \frac{1}{n} \sum_{i=1}^n (S_{i,d} - \IE S_{i,d}) + b_d \frac{1}{n} \sum_{i=1}^n \xi_{i,d}.
	\end{align*}
	We also write $V(d) = V_S(d) + b_d^2 V_\xi$ with $V_S(d) = \lb c_1 \frac{\IE S_{1,d}^2}{n} + c_2 \frac{d}{n} \rb \log n$ and $V_\xi = \frac{2 c_1}{n} \log n$. Hence,
	\begin{multline*}
		\IE  \lcb  |f_d(t) - \hat{f}_d(t)|^2 - \frac{V(d)}{3}  \rcb_+  \\ 
		\leq 2 \IE \lcb \bigg|\frac{1}{n} \sum_{i=1}^n (S_{i,h} - \IE S_{i,h})\bigg|^2 - \frac{V_S(d)}{6} \rcb_+  + 2 b_d^2 \IE \lcb \lv \frac{1}{n} \sum_{i=1}^n \xi_{i,d} \rv^2 - \frac{V_\xi}{6} \rcb_+. 
	\end{multline*}
	The claim follows from \cref{pde:lemma:conc:S,lemma:conc:xi} combined with \eqref{eq:sum}, $\lv \mc D \rv \leq n$ and $b_d^2 \leq 4 \phi_o^2\alpha^{-2} n^4$. 
\end{proof}

\begin{lemma}[Concentration of $\hat{V}$]
  There is a universal constant $C$ such that 
  \label{pde:lemma:conc:V}
  \begin{equation*}
    \IE \max_{d \in \mc D }
    \lcb V(d) - \hat{V}(d) \rcb_+  \leq  \frac{C \phi_o^2}{n\alpha^{2}}.
  \end{equation*}
\end{lemma}
\begin{proof}The proof follows line by line the proof
  of \cref{lemma:conc:V}. Similar to \eqref{eq:max} for $m \in \IN$,
  to be chosen below, we conclude
\begin{align}
	\label{pde:eq:max}
		\IE	 \max_{d \in \mc D} \lcb V(d) - \hat{V}(d) \rcb_+ \leq 2 c_1 (\log n)n^{-1} \sum_{d \in \mc D} \sigma^2_d 2^{m-1}\IE \lb  \lv \tfrac{\hat{\sigma}^2_d }{\sigma^2_d} - 1 \rv^{m} \rb .
\end{align}
Moreover, we have  $\hat{\sigma}^2_d \sigma^{-2}_d - 1  =
\frac{1}{n} \sum_{i=1}^n U_{i,d}$, where $U_{i,d}:=
Z_{i,d}^2\sigma^{-2}_d - 1$, $i \in \lcb 1 ,\dots, n \rcb$ are
independent random variables with $\IE U_{i,d} = 0$. Consequentley,
making again use of \eqref{petrov} we cobtain for $m \geq 2$
\begin{align}
	\label{pde:petrov}
	\IE \lb  \lv \tfrac{\hat{\sigma}^2_d }{\sigma^2_d} - 1 \rv^{m}
  \rb = \IE  \lv \frac{1}{n} \sum_{i=1}^n U_{i,d} \rv^{m} \leq c_m  n^{-m/2 - 1} \sum_{i=1}^n\lb 1 + \IE \lv \tfrac{Z_{i,d}}{\sigma^2_d} \rv^{m} \rb
\end{align}
for a constant $c_m$ only depending on $m$. Since 
\begin{equation*}
  \IE S_{1,d}^{2m} = \IE g_{d}^{2m}(X_1) \leq \IE\big(
  \sum_{j=1}^d\varphi_j^2(X_1)\sum_{l=1}^d\varphi_l^2(t)\big)^m \leq
  \IE \bigg(\bigg\Vert\sum_{j=1}^d\varphi_j^2\bigg\Vert_\infty^{2m}\bigg)\leq d^{2m}\phi_o^{2m}
\end{equation*}
as in the proof of \cref{lemma:conc:V}, from the bound
\eqref{eq:z:sigma} we obtain (exploiting $\sigma^2_d=\IE(S_{1,d}^2)+2b_d^2\geq 2 b_d^2$)
\begin{align*}
	\IE \lv \tfrac{Z_{i,d}}{\sigma^2_d} \rv^{m} \leq 2^m \frac{  d^{2m}\phi_o^{2m} + b_d^{2m} (2m)! }{\sigma_d^{2m}} \leq d^{2m}\phi_o^{2m}b_d^{-2m} + (2m)!,
\end{align*}
which together with $b_d^2 = 4 |\mc D|^2 \phi_o^2d^2\alpha^{-2}$ and
$\alpha/ \lv\mc D \rv \leq 1$ implies for all $d\in\mc D$
\begin{align*}
	\IE \lv \tfrac{Z_{i,d}}{\sigma^2_d} \rv^{m} \leq  4^{-2m} + (2m)!.
\end{align*}
Combining the last bound and \eqref{pde:petrov} there is a constant
$C_m$ only depending on $m$ such that 
\begin{align*}
		\IE \lb  \lv \tfrac{\hat{\sigma}^2_d }{\sigma^2_d} - 1 \rv^{m} \rb   \leq C_m n^{-m/2},
\end{align*}
Inserting this bound into \eqref{pde:eq:max} we obtain
\begin{align*}
		\IE	 \max_{d \in \mc D} \lcb V(d) - \hat{V}(d) \rcb_+ \leq C_m (\log n) n^{-1-m/2}  \sum_{d \in \mc D} \sigma^2_d
\end{align*}
Since $nd^{-1} \geq 1$ and $\lv \mc D \rv \leq n$ we have
\begin{align*}
	\sigma^2_d \leq d^2\phi_o^2 + 8\phi_o^2  \lv \mc H \rv^2 d^2\alpha^{-2} \leq (n^2  + 8 n^4\alpha^{-2})\phi_o^2,
\end{align*}
i.e. $n^{-4} \sigma^2_d \leq \phi_o^2 \lb 1 + 8\alpha^{-2} \rb$. Hence with $m = 12$ we obtain the claim. 
\end{proof}
\pagebreak
\begin{lemma}[Concentration for $\sum_{i=1}^nS_{i,d}$]\label{pde:lemma:conc:S}\leavevmode\\
	Recall the definition $V_S(d) = \lb c_1 n^{-1}\IE S_{1,d}^2 + c_2 dn^{-1} \rb \log n$. We have
	\begin{align*}
		\IE \lcb \lv \frac{1}{n} \sum_{i=1}^n (S_{i,d} - \IE S_{i,d})\rv^2 - \frac{V_S(d)}{6} \rcb_+ \leq 128 \phi_o^2n^{-2}.
	\end{align*}
\end{lemma}
\begin{proof}[Proof of \cref{pde:lemma:conc:S}]
	We have 
	\begin{multline}
		 \IE \lcb \lv \frac{1}{n} \sum_{i=1}^n (S_{i,d} - \IE
                 S_{i,d})\rv^2 - \frac{V_S(d)}{6} \rcb_+
                 \\ = \int_{0}^\infty \IP \lb  \lv   \frac{1}{n} \sum_{i=1}^n (S_{i,d} - \IE S_{i,d}) \rv \geq \sqrt{\frac{V_S(d)}{6}  + x} \rb \dif x. \label{pde:integral}
	\end{multline}
	We note that $S_{i,d}$, ${i \in \lcb 1, \dots, n \rcb}$, are
        i.i.d. with
        $\lv S_{1,d} \rv^2 \leq d\phi_o^2 =: \mathrm{b}$ and $\IV(S_{1,d}) \leq \IE S_{1,d}^2 =: \mathrm{v}^2$. We can write $V_S(d) = \lb c_1 \frac{\mathrm{v}^2}{n} + \frac{c_2}{\phi_o} \frac{\mathrm{b}}{n} \rb \log n$ Consequently, the Bernstein type inequality 	\cref{bernstein:ineq} yields for $\eps = \sqrt{\frac{V_S(d)}{6} + x}$ the following
	\begin{multline*}
		 \IP \lb \lv \frac{1}{n} \sum_{i=1}^n (S_{i,d} - \IE S_{i,d}) \rv \geq  \sqrt{\frac{V_S(d)}{6} + x}  \rb \\
		 \leq 2 \max \lcb \exp\lb - \frac{n}{4 \mathrm v^2} \lb \frac{V_S(d)}{6} + x  \rb \rb, \exp \lb - \frac{n}{4 \mathrm b} \sqrt{ \frac{V_S(d)}{6} + x} \rb \rcb \\
		 \leq 2 \exp(-3 \log n) \max \lcb \exp\lb - \frac{n}{4 \mathrm v^2} x  \rb, \exp \lb - \frac{n}{4 \sqrt{2} \mathrm b} \sqrt{ x} \rb \rcb,
	\end{multline*}
	since  $\sqrt{2} \sqrt{x+y} \geq \sqrt{x} + \sqrt{y}$ for $x,y \geq 0$ and $n d^{-1} \geq \log n$. Finally, we obtain
	\begin{equation}
		 \IP \lb \lv \frac{1}{n} \sum_{i=1}^n (S_{i,d} - \IE S_{i,d}) \rv \geq  \sqrt{\frac{V_S(d)}{6} + x}  \rb
		 \leq \frac{2}{n^3} \max \lcb \exp(-\tau_1 x), \exp(-\tau_2 \sqrt{x})\rcb. \label{pde:prop:bound}
	\end{equation}
	with $\tau_1 = \frac{n}{4 \mathrm{v}^2}$ and $\tau_2 =\frac{n}{4 \sqrt{2} \mathrm{b}} $. 
	Note that we have $\int_{0}^\infty \exp(-\tau_1 x) \dif x = \tfrac{1}{\tau_1} \leq 4 \phi_o^2 n$ and $\int_{0}^\infty \exp(-\tau_2 \sqrt{x}) \dif x = \frac{2}{\tau^2} \leq 4^3 \phi_o^2 n$, where we again exploited that $n d^{-1} \geq 1$. Hence, inserting the bound \eqref{pde:prop:bound} into \eqref{pde:integral} and integrating yields the result. 
\end{proof}
\section{Appendix}
\begin{theorem}[Bernstein inequality, \cite{Comte2017}, Appendix B, Lemma B.2]
	\label{bernstein:ineq}
	Let $(U_i)_{i \in \lcb 1, \dotsc, n \rcb}$ be independent and identically distributed random variables and let $\mathrm{v}^2$ and $\mathrm{b}$ be such that 
	$	\var(U_1) \leq \mathrm v^2$ and $\lv U_1 \rv \leq \mathrm b$. Then, for $\eps > 0$ we have
	\begin{align*}
		\IP\lb \lv \frac{1}{n} \sum_{i=1}^n \lb U_i - \IE U_i \rb \rv  \geq \eps \rb 
		&  \leq 2 \max \lcb  \exp \lb- \frac{n \eps^2}{4 \mathrm v^2} \rb, \exp\lb - \frac{n \eps}{4 \mathrm b} \rb \rcb.
	\end{align*}
\end{theorem}

\begin{proposition}[Petrov's inequality, \cite{Petrov1995}, Theorem 2.10]
	\label{petrovs:ineq}
	Let $(U_i)_{i \in \lcb 1, \dotsc, n \rcb}$ be independent with expectation zero. Then for $m \geq 2$ there exists a constant $c_m$ only depending on $m$ such that 
	\begin{align*}
		\IE \lb \lv \sum_{i=1}^n U_i \rv^m \rb \leq c_m n^{m/2 - 1} \sum_{i=1}^n \IE \lb \lv U_i \rv^m \rb.
	\end{align*}
\end{proposition}
\begin{proposition}[Tail bound for Laplace random variables]
	\label{prop:laplace}
	Let $(\xi_i)_{i \in \lcb 1, \dots, n \rcb}$ be independent and identically $\text{Laplace}(0,b)$-distributed random variables. Then, for all $\eps \geq 0$ we have
	\begin{align*}
		\IP\lb \lv \frac{1}{n} \sum_{i=1}^n \xi_i \rv \geq \eps  \rb \leq 2 \max \lcb \exp\lb - \frac{n \eps^2}{16 b^2}\rb, \exp\lb - \frac{n \eps}{2 b} \rb \rcb.
	\end{align*}
\end{proposition}

\begin{proof}[Proof of \cref{prop:laplace} ]
	We aim to apply Lemma 8.2. in \cite{Birge2001}. Straightforward calculations show that for $\xi_i \sim \text{Laplace}(0,b)$ and $0 < t < \frac{1}{b}$ we have
		$ \IE \exp(t\xi_i) = \frac{1}{1-b^2t^2}$. 
	Hence, it follows
	\begin{align*}
		\log \lb \IE \lb \exp(t\xi_i)\rb \rb = \log\lb \frac{1}{1-t^2b^2}\rb = \log \lb 1 + \frac{t^2 b^2}{1 - t^2b^2}\rb \leq \frac{t^2b^2}{1 - t^2b^2} \leq \frac{t^2b^2}{1 - t b},
	\end{align*}
	which implies for $S = \sum_{i=1}^n \xi_i$ that
	\begin{align*}
		\log \lb \IE \lb \exp(tS) \rb\rb = \sum_{i=1}^n \log \lb \IE \lb \exp(t\xi_i)\rb \rb \leq \frac{nt^2 b^2}{1 - tb}.
	\end{align*}
	Thus, we can apply Lemma 8.2. from \cite{Birge2001} with $ a = \sqrt{n} b$ and $b=b$. We obtain 
	\begin{align*}
		\IP \lb  \frac{1}{n} \sum_{i=1}^n \xi_i  \geq \frac{2b}{\sqrt{n}} \sqrt{x} + \frac{b}{n} x\rb \leq  \exp(-x).
	\end{align*}
	The claim now follows from the symmetry of the centered Laplace distribution and simple rearranging of this inequality. 
\end{proof}

%
%
%
%
\clearpage
\addcontentsline{toc}{part}{Bibliography}
\bibliography{lit.bib}
\end{document}